\documentclass[reqno,10pt]{amsart}
\usepackage{amssymb}
\usepackage{latexsym}
\usepackage{hyperref}
\usepackage{amsmath}
\usepackage{amscd}
\usepackage{cite}
\usepackage{color}
\usepackage{enumerate}
\usepackage{amsfonts}
\usepackage{graphicx}
\usepackage{mathrsfs}
\usepackage{setspace}
\newcommand{\eps}{{\varepsilon}}
\theoremstyle{plain}
\newtheorem{theorem}{Theorem}
\newtheorem{proposition}[theorem]{Proposition}
\newtheorem{lemma}[theorem]{Lemma}

\theoremstyle{definition}
\newtheorem{definition}[theorem]{Definition}
\newtheorem{remark}[theorem]{Remark}

\newcommand{\qtq}[1]{\quad\text{#1}\quad}
\numberwithin{equation}{section}
\numberwithin{theorem}{section}
\let\Re=\undefined\DeclareMathOperator*{\Re}{Re}
\let\Im=\undefined\DeclareMathOperator*{\Im}{Im}

\newcommand{\F}{\mathcal{F}}
\def\ge{\geqslant}
\def\le{\leqslant}
\def\geq{\geqslant}
\def\leq{\leqslant}

\def\R{\mathbb{R}}
\def\C{\mathbb{C}}
\def\eps{\varepsilon}

\begin{document}

\title[4D cubic NLS]{Dynamics of subcritical threshold solutions \\ for the 4$d$ energy-critical NLS}

\author[Z. Ma]{Zuyu Ma}
\address{Graduate School of China Academy of Engineering Physics, Beijing 100088, China}
\email{mazuyu23@gscaep.ac.cn}

 \author[C. Miao]{Changxing Miao}
\address{Institute of Applied Physics and Computational Mathematics and National Key Laboratory of Computational Physics, Beijing 100088, China}
\email{miao\_changxing@iapcm.ac.cn }

\author[J. Murphy]{Jason Murphy}
\address{Department of Mathematics, University of Oregon, Eugene, OR, USA}
\email{jamu@uoregon.edu}

\author[J. Zheng]{Jiqiang Zheng}
\address{Institute of Applied Physics and Computational Mathematics and National Key Laboratory of Computational Physics, Beijing 100088, China}
\email{zheng\_jiqiang@iapcm.ac.cn}

\maketitle 

\begin{abstract}
We study dynamics of the 4$d$ energy-critical nonlinear Schr\"odinger equation at the ground state energy. Previously, Duyckaerts and Merle [Geom. Funct. Anal. (2009)] proved that any radial solution with kinetic energy less than that of the ground state either scatters in both time directions or coincides (modulo symmetries) with a heteroclinic orbit, which scatters in one time direction and converges to the ground state in the other.  We extend this result to the non-radial setting.
\end{abstract}

%%%%%%%%%%%%%%%%%%%%%%%%%%%%%%%%%%%%%%%%%%%%%%%%%%%%%%%%%%%%%%
%%%%%%%%%%%%%%%%%%%%%%%%%%%%%%%%%%%%%%%%%%%%%%%%%%%%%%%%%%%%%%
%%%%%%%%%%%%%%%%%%%%%%%%%%%%%%%%%%%%%%%%%%%%%%%%%%%%%%%%%%%%%%
%%%%%%%%%%%%%%%%%%%%%%%%%%%%%%%%%%%%%%%%%%%%%%%%%%%%%%%%%%%%%%
\section{Introduction}

We consider the $4d$ focusing energy-critical nonlinear Schr\"odinger equation (NLS)
\begin{equation}\label{1.1}
\begin{cases}
(i\partial_t+\Delta) u=-|u|^2u,\quad (t,x)\in\R\times\R^4,\\
u|_{t=0}=u_0\in \dot{H}^1(\R^4),
\end{cases}
\end{equation}
which is a special case of the more general energy-critical equation
\begin{align} \label{equ:encrinls}
\begin{cases}
(i\partial_t+\Delta)u= -|u|^{\frac4{d-2}} u,\quad (t,x)\in\R\times\R^d,\\
u|_{t=0}=u_0\in \dot{H}^1(\R^d),
\end{cases}
\end{align}
in dimensions $d\geq 3$.  Equation \eqref{equ:encrinls} is the Hamiltonian flow corresponding to the (conserved) energy
\begin{equation}\label{equ:energydef}
E(u(t))=\tfrac12\int_{\R^d}|\nabla u(t,x)|^2\;dx-\tfrac{d-2}{2d}\int_{\R^d}|u(t,x)|^\frac{2d}{d-2}\;dx.
\end{equation}

This equation admits a stationary solution in $\dot H^1$ known as the \emph{ground state}, given explicitly by 
\begin{equation}
\label{defW}
W(x):= \bigl(1+\tfrac{|x|^2}{d(d-2)}\bigr)^{-\frac{d-2}{2}}.
\end{equation}
The ground state satisfies the elliptic equation
\begin{equation}\label{equ:CP}
\Delta W+|W|^\frac{4}{d-2}W=0
\end{equation}
and admits a variational characterization in terms of the sharp Sobolev embedding estimate.  In particular, writing $C_d$ for the sharp constant in dimension $d$, we have
\begin{gather}
\label{SobolevIn}
\|u\|_{L^\frac{2d}{d-2}(\R^d)}\leq C_d\|u\|_{\dot{H}^1(\R^d)}\qtq{for all}u\in \dot H^1(\R^d),
\end{gather}
with equality if and only if 
\[
u(x)=z_0  W\big(\lambda({x+y})\big)
\]
for some $(\lambda,z_0,y)\in(0,\infty)\times\C\times\R^d$ (see \cite{Au76, Ta76}).

The ground state plays an important role in classifying the dynamics of solutions to \eqref{equ:encrinls}.  For solutions with energy below that of the ground state, one has the following theorem:

\begin{theorem}[Dynamics below ground state energy, \cite{kenig-merle1, KV1, dodson1}]\label{thm:KM} Let $d\geq 3$ and $u_0\in \dot H^1(\R^d)$ such that $E(u_0)<E(W)$. We also  assume that $u_0$ is radial if $d=3$.
\begin{itemize}
\item[(i)] If $\|\nabla u_0\|_{L^2}<\|\nabla W\|_{L^2}$, then the solution to \eqref{equ:encrinls} with initial data $u_0$ is global and scatters\footnote{We say $u$ scatters as $t\to\pm\infty$ if there exist $u_\pm\in \dot H^1$ such that $\lim_{t\to\pm\infty}\|u(t)-e^{it\Delta}u_\pm\|_{\dot H^1}=0$, where $e^{it\Delta}$ is the free Schr\"odinger propagator.} in both time directions.
\item[(ii)] If $u_0\in H^1$ is radial and $\|\nabla u_0\|_{L^2}>\|\nabla W\|_{L^2}$, then the solution to \eqref{equ:encrinls} with initial data $u_0$ blows up in finite time in both time directions. 
\end{itemize}
\end{theorem}
These results were first obtained by Kenig and Merle in \cite{kenig-merle1} in dimensions $d\in\{3,4,5\}$ for radial data.  Killip and Vi\c{s}an \cite{KV1} subsequently treated the non-radial case in dimensions $d\geq 5$, and Dodson \cite{dodson1} addressed the non-radial case in $d=4$.  The non-radial case for $d=3$ remains a major open problem in the field.  An essential difficulty stems from the fact that while $W\in H^1$ for $d\geq 5$, one only has $W\in \dot H^1 \cap L^{2,\infty}$ in $d=4$ and $W\in \dot H^1\cap L^{3,\infty}$ in $d=3$. Let us also remark that in the defocusing case, one obtains scattering for all finite energy solutions (see \cite{Visan3, Tao1, Visan2, Ryckman, KV2, bourgain, CKSTT}).

Dynamics at the ground state energy have also been classified, albeit primarily in the radial setting\footnote{In fact, the work \cite{SuZhao} is related to threshold dynamics in the non-radial case in dimensions $d\geq 5$.  We will discuss this reference below.} \cite{duyckaerts1, li1, campos}. The classification involves two additional special solutions, namely, heteroclinic orbits denoted by $W^\pm\in \dot H^1$.  These solutions satisfy $E(W^\pm)=E(W)$ and behave as follows:  the solution $W^-$ scatters as $t\to-\infty$ and converges to $W$ in $\dot H^1$ as $t\to+\infty$, while $W^+$ blows up in finite time in the negative time direction and converges to $W$ in $\dot H^1$ as $t\to+\infty$.

\begin{theorem}[Radial dynamics at ground state energy \cite{duyckaerts1, li1, campos}]\label{thm:claener}
Let $d\geq 3$. Let $u_0\in \dot{H}^1$ be radial and satisfy $E(u_0)=E(W)$, and let $u:I\times\R^d\to\C$ be the maximal-lifespan solution to \eqref{equ:encrinls} with $u|_{t=0}=u_0$.
\begin{enumerate}
\item[(i)] If $\|\nabla u_0\|_{L^2}<\|\nabla W\|_{L^2}$, then $I=\R$. Furthermore, either $u$ scatters in both time directions or $u=W^-$ modulo symmetries.
\item[(ii)] If $\|\nabla u_0\|_{L^2}=\|\nabla W\|_{L^2}$, then $u=W$ modulo symmetries.
\item[(iii)] If $u_0\in L^2$ and $\|\nabla u_0\|_{L^2}>\|\nabla W\|_{L^2}$, then either $I$ is finite or $u=W^+$ modulo symmetries. 
\end{enumerate}
\end{theorem}

We may refer to cases (i)--(iii) as the subcritical, critical, and supercritical cases, respectively.  When we say $u=v$ modulo symmetries, we mean that
\[
u(t,x)=e^{i\theta}\lambda^{-\frac{d-2}2}v\big(\tfrac{t-s}{\lambda^2},\tfrac{x-y}{\lambda}\big)\qtq{or} u(t,x)=e^{i\theta}\lambda^{-\frac{d-2}2}\bar{v}\big(\tfrac{s-t}{\lambda^2},\tfrac{x-y}{\lambda}\big)
\]
for some $(\lambda,\theta,y,s)\in(0,\infty)\times\R\times\R^d\times\R$. Note that one always has $y=0$ in the radial setting.  We also remark that case (ii) does not require a radial assumption, in light of the variational characterization of $W$ described above. We also refer to \cite{campos,duyckaerts4} for the energy-subcritical NLS and \cite{dodson5,dodson6} for the mass-critical NLS.

In this work, we remove the radial assumption in case (i) in dimension $d=4$.  In particular, our main result may be stated as follows. 

\begin{theorem}[Subcritical dynamics at ground state energy] \label{t1.1}
Let $d=4$ and $u_0\in \dot H^1$ with $E(u_0)=E(W)$ and $\|\nabla u_0\|_{L^2}<\|\nabla W\|_{L^2}$. 

Let $u:I\times\R^4\to\C$ be the maximal-lifespan solution to \eqref{1.1} with $u|_{t=0}=u_0$. 

Then $I=\R$, and either $u$ scatters in both time directions or $u=W^-$ modulo symmetries.
\end{theorem}

The proof of Theorem~\ref{t1.1} employs the general framework introduced in \cite{duyckaerts1}, leveraging the analysis of \cite{dodson1} in order to address the non-radial case.  In the appendix, we will also discuss the extension of Theorem~\ref{t1.1} to dimensions $d\geq 5$.  In fact, this problem has previously been studied in high dimensions in \cite{SuZhao}, although we contend that this work contains some gaps that must be filled in order to complete the proof (e.g. the solution $W^-$ is a counterexample to Theorem~2.10 therein).  In Appendix~\ref{S:highd}, we therefore describe some of the key ingredients necessary for the extension to dimensions $d\geq 5$.  The high-dimensional problem is ultimately simpler due to the fact that the ground state belongs to $L^2$ provided $d\geq 5$.  Correspondingly, one can prove better decay properties for almost periodic solutions, leading to a more direct implementation of virial-type arguments.  On the other hand, the nonlinearity becomes less regular in high dimensions, which does lead to some technical challenges.  Fortunately, it is known how to deal with such issues (cf. \cite{Visan2, KV1, li1,campos, Tao1, SuZhao}).

Before moving on to the outline of the proof of Theorem~\ref{t1.1}, let us discuss one other related problem, namely, the characterization of the upper bounds of scattering norms as one approaches the ground state energy.  Following Duyckaerts and Merle \cite{duyckaerts}, we may consider the following quantity for any $\eps>0$:
\[
\mathcal{I}_\eps = \sup_{u\in F_\eps} \iint |u(t,x)|^{\frac{2(d+2)}{d-2}}\,dx\,dt,
\]
where $F_\eps$ is the set of solutions to \eqref{equ:encrinls} such that 
\[
E(u_0)\leq E(W)-\eps^2\qtq{and}\|\nabla u_0\|_{L^2}<\|\nabla W\|_{L^2}.
\]
In the case $d\in\{3,4,5\}$ and radial solutions, \cite[Theorem~2]{duyckaerts} shows that
\[
\lim_{\eps\to 0+} \tfrac{\mathcal{I}_\eps}{|\log\eps|} = \tfrac{2}{\lambda_1}\int |W|^{\frac{2(d+2)}{d-2}}\,dx,
\]
where $\lambda_1$ is the unique positive eigenvalue of the linearized operator $\mathcal{L}$ (cf. \eqref{defL} below).  The analysis in the present paper extends this result to the non-radial setting in dimensions $d\geq 4$. 

%%%%%%%%%%%%%%%%%%%%%%%%%%%%%%%%%%%%%%%%%%%%%%%
%%%%%%%%%%%%%%%%%%%%%%%%%%%%%%%%%%%%%%%%%%%%%%%
%%%%%%%%%%%%%%%%%%%%%%%%%%%%%%%%%%%%%%%%%%%%%%%

\subsection{Outline proof of Theorem \ref{t1.1}} 

In this subsection, we sketch the proof of Theorem~\ref{t1.1} and outline the structure of the paper.  Throughout this subsection, we always consider solutions $u$ satisfying $E(u)=E(W)$ and $\|\nabla u_0\|_{L^2}<\|\nabla W\|_{L^2}$. 

We first establish global well-posedness (Proposition~\ref{T:no ftb}). This relies on the ever useful Lemma~\ref{alcompact}, which shows that due to the sub-threshold scattering result, blowup (in the sense of infinite scattering norm) guarantees almost periodicity.  In particular, a finite-time blowup solution must be almost periodic.  Moreover, to blow up in finite time it must escape to arbitrarily small spatial scales, in the sense that its frequency scale function must satisfy $N(t)\to\infty$.  For such a solution, we can apply an argument based on the localized mass to prove that the solution is in fact an $L^2$ solution, with mass equal to \emph{zero}.  This is a contradiction, from which we conclude global existence.     

To prove Theorem~\ref{t1.1}, it therefore suffices to show that any solution that blows up (in the sense of infinite scattering norm) in at least one time direction must equal $W^-$ modulo symmetries. A key step is taken in Proposition~\ref{sequence}, which states that for such a solution, there must exist a sequence of times $\{t_n\}\subset\R$ such that
\begin{equation}\label{intro-delta1}
\delta(u(t_n)):=\bigl| \|\nabla W\|_{L^2}^2 - \|\nabla u(t_n)\|_{L^2}^2 \bigr| \to 0. 
\end{equation}
Using Lemma~\ref{alcompact} and some standard arguments related to well-posedness, we can reduce matters to proving that this is the case for a solution that is almost periodic on $\R$ with frequency scale function $N(t)$ satisfying $\inf_{t\in\R}N(t)\geq 1$.  

At this point, we appeal to the work of Dodson concerning solutions with energy below that of the ground state \cite{dodson1}.  We consider two possible scenarios based on the behavior of $N(t)$, which are called the \emph{frequency cascade} and \emph{quasi-soliton} scenarios.  The analysis of \cite{dodson1} shows that the cascade case is impossible, independent of the size of the energy.  Thus we only need to consider the quasi-soliton case.  In this case, Dodson proves via an interaction Morawetz estimate that if the \emph{kinetic} energy stays strictly below that of the ground state, the solution must vanish identically (see Proposition~\ref{P:nscqs}).  Applying this result in our setting, we conclude that there must be some sequence of times along which the kinetic energy converges to that of $W$, yielding \eqref{intro-delta1}.  It is a somewhat subtle point that the argument depends only on the kinetic energy assumption (rather than the full energy), so we include in Section~\ref{S:Dodson} an abbreviated presentation of Dodson's interaction Morawetz estimate. 

After collecting some preliminaries related to modulation analysis in Section~\ref{S:modulation}, the next key ingredient is a type of rigidity result, Theorem~\ref{t2.3}.  This result states that if a solution stays sufficiently close to the ground state in one time direction (as measured by the functional $\delta(u(t))$), then $u$ must equal $W^-$ modulo symmetries.  With this result in hand, we complete the proof of Theorem~\ref{t1.1} as follows.  If the solution stays sufficiently close to ground state, Theorem~\ref{t2.3} applies directly to conclude the proof.  If instead the solution continues to maintain some distance from the ground state along some sequence $t_n'\to\infty$, then by utilizing the two sequences $t_n$ (along which $\delta(u(t_n))\to 0$) and $t_n'$ (along which $\delta(u(t_n'))\geq\eta_*>0$) we can extract a solution that stays very close to the ground state for all times yet blows up in both time directions. However, the existence of such a solution contradicts Theorem~\ref{t2.3} and the fact that $W^-$ scatters in one time direction.  Thus this second alternative cannot occur, and we so complete the proof of Theorem~\ref{t1.1}.

It remains to describe the proof of Theorem~\ref{t2.3}.  The radial analogue of this result essentially appears in \cite{duyckaerts1} and plays a similar central role in that work.  We follow the general strategy of \cite{duyckaerts1}, making the necessary modifications to address the non-radial setting.  The first step is to utilize the `modulated' virial estimate, a version of the localized virial estimate that takes into account the fact that the solution stays near the orbit of the ground state.  Using this estimate, we firstly establish that the modulation parameters are nearly constant.  After establishing this fact, another application of the modulated virial estimate (together with the fact that $\delta(u(t_n))\to 0$ along some sequence $t_n\to\infty$) implies exponential decay of $\delta(u(t))$ (and exponential convergence of the modulation parameters).  This is the essential input needed to apply the rigidity result of \cite{duyckaerts1}, stating roughly that exponential convergence to the ground state guarantees that the solution must be exactly $W^-$ (modulo symmetries).  The necessary result is stated as Proposition~\ref{2WW}, which appears exactly in \cite{duyckaerts1} in the radial case.  In Section~\ref{S:2WW}, we sketch the strategy of the proof in \cite{duyckaerts1}, pointing out the changes necessary to address the non-radial case. 

Finally, as mentioned above, the extension to dimensions $d\geq 5$ is outlined in Appendix~\ref{S:highd}. 

\subsection*{Acknowledgements} C. Miao was supported by the National Key RR\&D program
of China: No.2021YFA1002500 and NSFC Grant 12026407. J.M. was supported by NSF grant DMS-2350225. J. Zheng was supported by  National key R\&D program of China: 2021YFA1002500 and  NSFC Grant 12271051.

%%%%%%%%%%%%%%%%%%%%%%%%%%%%%%%%%%%%%%%%%%%%%%%%%%%%%%%%%%%%%%
%%%%%%%%%%%%%%%%%%%%%%%%%%%%%%%%%%%%%%%%%%%%%%%%%%%%%%%%%%%%%%
%%%%%%%%%%%%%%%%%%%%%%%%%%%%%%%%%%%%%%%%%%%%%%%%%%%%%%%%%%%%%%
%%%%%%%%%%%%%%%%%%%%%%%%%%%%%%%%%%%%%%%%%%%%%%%%%%%%%%%%%%%%%%
\section{Notation and useful lemmas}

In this section, we collect some analysis tools and some fundamental results needed throughout the paper.
%%%%%%%%%%%%%%%%%%%%%

\subsection{Notation} For nonnegative quantities $X$ and $Y$, we write $X\lesssim Y$ to denote the estimate $X\leq C Y$ for some $C>0$. If $X\lesssim Y\lesssim X$, we will write $X\sim Y$. Dependence on various parameters will be indicated by subscripts, e.g. $X\lesssim_u Y$ indicates $X\leq CY$ for some $C=C(u)$. 

We will make use of the $L^2$ and $\dot H^1$ inner products given by $(f,g)_{L^2}=\Re\int f\bar g$ and $(f,g)_{\dot H^1}=(\nabla f,\nabla g)_{L^2}$. We will use $A^\perp$ to denote the orthogonal complement of a set $A$. 

For a spacetime slab $I\times\R^4$, we write $L_t^q L_x^r(I\times\R^4)$ for the Banach space of functions $u:I\times\R^4\to\C$ equipped with the norm
\[
\|u\|_{L_t^{q}L_x^r(I\times\R^4)}:=\bigg(\int_I \|u(t)\|_{L_x^r(\R^4)}\,dt\bigg)^{1/q},
\]
with the usual adjustments if $q$ or $r$ is infinity. When $q=r$, we abbreviate $L_t^qL_x^q=L_{t,x}^q$. We also abbreviate $\|f\|_{L_x^r(\R^4)}$ to $\|f\|_{L_x^r}.$ 

The Fourier transform on $\mathbb{R}^4$ is given by
\begin{equation*}
\aligned \widehat{f}(\xi)= \F f(\xi)= \tfrac{1}{4\pi^2}\int_{\mathbb{R}^4}e^{- ix\cdot \xi}f(x)\,dx.
\endaligned
\end{equation*}

% We denote the standard Littlewood--Paley projection operators by $P_{\leq N}$, $P_{>N}$, $P_N$, $P_{M<\cdot\leq N}$, and so on, where we always take $N\in 2^{\mathbb{Z}}$. For the basic properties of these operators, including the extremely useful Bernstein estimates, we refer the reader to \cite{VisanOberwolfach}. 

\subsection{Some analysis tools} We denote the free Schr\"odinger propagator by $e^{it\Delta}$.  It is most naturally defined as the Fourier multiplier operator with symbol $e^{-it|\xi|^2}$, but also has a physical-space representation given by convolution with a complex Gaussian. This operator satisfies well-known estimates known as Strichartz estimates, which we state as follows.

First, we call a pair of exponents $(q,r)$ \emph{Schr\"odinger admissible} (in four dimensions) if $2\leq q,r\leq\infty$ and $\tfrac{2}{q}+\tfrac{4}{r}=2.$ For a spacetime slab $I\times\R^4$, we define the Strichartz norm 
\[
\|u\|_{S^0(I)}:=\sup\big\{\|u\|_{L_t^{q}L_x^r(I\times\R^4)}:(q,r)\text{ is Schr\"odinger admissible}\big\}.
\]
We denote $S^0(I)$ to be the closure of all test functions under this norm and write $N^0(I)$ for the dual of
$S^0(I)$. We then have the following:

\begin{proposition}[Strichartz estimates, \cite{GinibreVelo, KeelTao, Strichartz}]\label{prop1}
Suppose $u:I\times\R^4\to\C$ is a solution to $(i\partial_t+\Delta)u=F$. For any $t_0\in I$, 
\[
\|u\|_{S^0(I)}\lesssim\| u(t_0)\|_{L_x^2(\R^4)}+\|F\|_{N^0(I)}.
\]
\end{proposition}

We next give some basic properties for the ground state $W$.
\begin{lemma}[Sharp Sobolev embedding, \cite{Au76, Ta76}] \label{GNsharpz}
For any $f\in \dot H^1(\R^4)$, 
\begin{equation}\label{GNinequality}
\|f\|_{L^{4}_x(\mathbb{R}^4)} \leq C_4 \|f\|_{\dot{H}^{1}(\mathbb{R}^4)},
\end{equation}
where 
\[
C_{4}:=\|W\|_{L^{4}_x(\mathbb{R}^4)} \|W\|_{\dot{H}^{1}(\mathbb{R}^4)}^{-1}.
\]
Moreover, equality holds in \eqref{GNinequality} if and only if \(f(x)=z_0  W(\lambda_0 x+x_0)\) for some  $\lambda_0>0$, $z_0\in\C$ and $x_0\in\R^4$.
\end{lemma}
The ground state $W$ also satisfies the following Pohozaev identities:
\begin{equation*}
\|W\|_{\dot{H}^1(\R^4)}^2=\|W\|^{4}_{L^4(\R^4)} \quad \mbox{and }\ \quad E(W)=\tfrac{1}{4}\|W\|_{\dot{H}^1(\R^4)}^2.
\end{equation*}

We will need the following two elementary lemmas related to energy trapping.  The first is from \cite{duyckaerts1}.
\begin{lemma}\label{varlemma}
Let $f\in \dot{H}^1(\mathbb{R}^4)$ such that $\|f\|_{\dot{H}^1}\leq \|W\|_{\dot{H}^1}$. Then
\begin{equation*}
\frac{\|f\|_{\dot{H}^1}^2}{\|W\|_{\dot{H}^1}^2}\leq \frac{E(f)}{E(W)}  
\end{equation*} 
and there exists constant $C>0$ such that
\begin{equation*}
\tfrac{1}{C}\|f\|_{\dot{H}^1(\R^4)}^2\leq E(f)\leq C\|f\|_{\dot{H}^1(\R^4)}^2.
\end{equation*}
In particular, $E(f)$ is positive.
\end{lemma}

Using this lemma, we may readily obtain the following: 

\begin{lemma}\label{GlobalS}
Suppose that $u_{0}\in \dot{H}^{1}(\R^4)$ satisfies
\begin{equation}\label{f2.5}
E(u_0)= E(W)\qtq{and} \|u_0\|_{\dot{H}^1}<\|W\|_{\dot{H}^1}.
\end{equation}
Then the solution to \eqref{1.1} with $u|_{t=0}=u_0$ satisfies 
\begin{equation}\label{f2.6}
\|u(t)\|_{\dot{H}_x^1}<\|W\|_{\dot{H}^1}
\end{equation}
throughout its maximal lifespan. 
\end{lemma}

\subsection{Local theory} We turn to local well-posedness and stability for \eqref{1.1}. An important role is played by the $L_{t,x}^6$ spacetime norm of solutions, which we refer to as the \emph{scattering norm} and denote by
\[
S_I(u) = \|u\|_{L_{t,x}^6(I\times\R^4)}^6.
\]  
In particular, solutions may be extended in time as long as this norm remains finite, and a global solution with finite $L_{t,x}^6$-norm necessarily scatters in $\dot H^1$ in both directions. In light of this, we define \emph{blowup} to mean infinite $L_{t,x}^6$-norm. We remark that this is different from the notion of finite-time blowup mentioned in the introduction (in particular, the present notion of blowup may occur in finite or infinite time).

\begin{definition}[Blowup]\label{s2.7def}
Let $u:I\times\R^4\to\C$ be a maximal-lifespan solution to \eqref{1.1}. We say $u$ \emph{blows up forward in time} if there exists a time $t_1\in I$ such that $S_{[t_1, \sup I)}(u)=\infty$, and that $u$ \emph{blows up backward in time} if there exists a time $t_1 \in I$ such that $S_{(\inf I, t_1]}(u) = \infty.$
\end{definition}

We then have the following local well-posedness result for \eqref{1.1}.

\begin{theorem}[Local well-posedness, \cite{KVnote, VisanOberwolfach}]\label{T:local}
Let $u_0\in\dot H_x^1(\R^4)$ and $t_0\in \R$. There exists a unique maximal-lifespan solution $u:I\times\R^4 \to \C$ to
\eqref{1.1} with initial data $u|_{t=t_0}=u_0$.  Furthermore, we have the following:
\begin{enumerate}
\item [$(i)$] $I$ is an open neighborhood of $t_0$.
\item [$(ii)$] If $\sup(I)$ is finite, then $u$ blows up forward in time; if $\inf(I)$ is finite,
then $u$ blows up backward in time.
\item[$(iii)$] If $\sup(I)=\infty$ and $u$ does not blow up forward in time, then $u$ scatters forward in time.  The analogous statement holds backward in time.
\item[$(iv)$] If $\|\nabla u_0\|_{L^2}$ is sufficiently small, then $u$ is global and scatters in both time directions, with $S_\R(u)\lesssim \|\nabla u_0\|_{L^2}^{6}$.
\end{enumerate}
\end{theorem}

In addition, we will need the following stability result for \eqref{1.1}.

\begin{lemma}[Stability, \cite{KVnote, VisanOberwolfach}]\label{pertu}  For any $E,L > 0$ and $\eps > 0$ there exists $\delta > 0$
such that the following holds: Suppose $\tilde u: I \times \R^4 \to \C$ satisfies
\begin{equation*}
\bigl\| \nabla \bigl[(i\partial_t+\Delta)\tilde u +|\tilde u|^2 \tilde u \bigr] \bigr\|_{L_{t,x}^{\frac{3}{2}}(I\times\R^4)} \leq \delta
\end{equation*}
as well as
\begin{align*}
\|\tilde u\|_{L_t^\infty \dot H^1_x(I\times\R^4)}\leq E \quad \text{and} \quad S_I(\tilde u)\leq L.
\end{align*}
If $t_0 \in I$ and $u_0 \in \dot H^1_x(\R^4)$ are such that
\[
\|\tilde u(t_0)-u_0\|_{\dot H^1_x(\R^4)} \leq \delta,
\]
then there exists a solution $u: I \times \R^4 \to \C$ to \eqref{1.1} with $u|_{t=t_0}= u_0$ such that
\begin{align*}
\|\tilde u -u\|_{L_t^\infty \dot H^1_x(I\times\R^4)} + S_I(\tilde u - u) \leq \eps.
\end{align*}
\end{lemma}

\subsection{Concentration-compactness and almost periodicity} We next record some concentration compactness tools and introduce the notion of an almost periodic solution.

We begin by introducing the symmetry group adapted to $\dot H^1(\R^4)$:
\begin{definition}[Symmetry group] For any $\lambda_0>0$ and $x_0\in\R^4$, we define a unitary transformation $g_{\lambda_0, x_0}$: $\dot{H}^1(\R^4)\to \dot{H}^1(\R^4)$ by
\begin{align*}
\left(g_{\lambda_0, x_0} f \right) (x) = \tfrac{1}{\lambda_{0}}  f 
\bigl( \tfrac{1}{\lambda_{0}}( x - x_0) \bigr).
\end{align*}
We let $G$ denote the group of all such transformations under composition.
\end{definition}

\begin{lemma} [Linear profile decomposition, \cite{keraani1,KVnote}]\label{L:cc}
Let $\{u_n\}_{n\geq 1}$ be a sequence of functions bounded in $\dot H_x^1(\R^4)$. After passing to a subsequence if necessary, there exist a sequence of functions $\{\phi^j\}_{j\geq 1}\subset \dot H_x^1(\R^4)$, group elements $g_n^j \in G$, and times $t_n^j\in \R$ such that we have the following decomposition
\begin{align}\label{profile1}
u_n = \sum_{j=1}^J g_n^j e^{it_n^j\Delta}\phi^j + w_n^{J}
\end{align}
for all $J\geq 1$. Here $w_n^{J} \in \dot H^1_x(\R^4)$ obey
\begin{equation}\label{w scat}
\lim_{J\to \infty}\limsup_{n\to\infty} \bigl\| e^{it\Delta}w_n^{J}\bigr\|_{L_{t,x}^{6}(\R\times\R^4)}=0.
\end{equation}
Moreover, for any $j \neq j'$,
\begin{align}\label{profile3}
\tfrac{\lambda_{n}^j}{\lambda_{n}^{j'}} + \tfrac{\lambda_{n}^{j'}}{\lambda_n^{j}} + \tfrac{|x_{n}^j-x_n^{j'}|^2}{\lambda_{n}^j \lambda_n^{j'}} + \tfrac{|t_{n}^j(\lambda_n^j)^2- t_{n}^{j'}(\lambda_{n}^{j'})^2|}{\lambda_{n}^j\lambda_{n}^{j'}}\to\infty
    \quad \mbox{ as }\quad n\to \infty.
\end{align}
Furthermore, for any $J \geq 1$ we have the kinetic energy decoupling
\begin{equation}\label{profile4}
\lim\limits_{n \to \infty} \biggl[ \|\nabla u_n\|^2_{L^2(\R^4)} - \sum_{j=1}^J \|\nabla\phi^j\|^2_{L^2(\R^4)} - \|\nabla w_n^J\|^2_{L^2(\R^4)} \biggr] = 0
\end{equation}
and $L^4$ decoupling
\begin{equation}\label{profile5}
\lim\limits_{n\to\infty} \biggl[ \|u_n\|_{L^4(\R^4)}^4 - \sum_{j=1}^J \| \phi^j\|_{L^4(\R^4)}^4 - \|w_n\|_{L^4(\R^4)}^4 \biggr] = 0.
\end{equation}
\end{lemma}

The main application of linear profile decompositions in our setting is to prove the existence of solutions to \eqref{1.1} obeying a certain compactness property known as almost periodicity. 

\begin{definition}[Almost periodic] A solution $u:I\times\R^4\to\C$ to \eqref{1.1} is \emph{almost periodic (modulo symmetries)} if  there exist functions $N:I\to\R^+$, $x:I\to\R^4$, and $C:\R^+\to\R^+$ such that for $t\in I$ and $\eta>0$,
\[
\int_{\left|x-x(t)\right|\geq\frac{C(\eta)}{N(t)}}|\nabla u(t,x)|^2\, dx+\int_{|\xi|\geq C(\eta)N(t)}\left|\xi\right|^2|\hat{u}(t,\xi)|^2\, dx\leq\eta.
\] 
We call $N(t)$ the \emph{frequency scale}, $x(t)$ the \emph{spatial center}, and $C(\eta)$ the \emph{compactness modulus}. 
\end{definition}

Using the embedding $\dot H^1\subset L^4$, we may also choose $C(\eta)$ so that
\begin{equation*}
\int_{\left|x-x(t)\right|\geq\frac{C(\eta)}{N(t)}}|u(t,x)|^4\, dx\leq \eta.
\end{equation*}

%%%%%%%%%%%%%%%%%%%%%%%%%%%%%%%%%%%%%%%%%%%%%%%%%
%%%%%%%%%%%%%%%%%%%%%%%%%%%%%%%%%%%%%%%%%%%%%%%%%
%%%%%%%%%%%%%%%%%%%%%%%%%%%%%%%%%%%%%%%%%%%%%%%%%
%%%%%%%%%%%%%%%%%%%%%%%%%%%%%%%%%%%%%%%%%%%%%%%%%
\section{Global well-posedness}

In this section, we establish the global existence component of Theorem~\ref{t1.1}.

\begin{lemma}\label{alcompact}
Suppose $u:I\times\R^4\to\C$ is a maximal-lifespan solution to \eqref{1.1} satisfying
\[
E(u)=E(W) \qtq{and} \|\nabla u_0\|_{L^2}< \|\nabla W\|_{L^2}.
\]
If $u$ blows up forward in time, then $u$ is almost periodic on $[0,\sup I)$.  Similarly, if $u$ blows up backward in time, then $u$ is almost periodic on $(\inf I,0]$.
\end{lemma}

\begin{proof}[Sketch of the proof] The argument to obtain this lemma follows along fairly standard lines (see e.g. \cite[Section~2.3]{duyckaerts1}), so we only provide a brief sketch here. The idea is to take an arbitrary sequence of times $t_n$ and apply the linear profile decomposition to $u(t_n)$. Using Theorem~\ref{thm:KM} and stability theory, we can use the construction of a \emph{nonlinear} profile decomposition to argue that there must be exactly one profile in the profile decomposition (since in all other cases we obtain scattering for the solution $u$). This yields the desired compactness.
\end{proof}

\begin{proposition}\label{T:no ftb} Suppose $u:I\times\R^4\to\C$ is a maximal-lifespan solution to \eqref{1.1} satisfying
\[
E(u)=E(W) \qtq{and} \|\nabla u_0\|_{L^2}\leq \|\nabla W\|_{L^2}.
\]
Then $I=\R$.
\end{proposition}

\begin{proof} Note that if $\|\nabla u_0\|_{L^2}=\|\nabla W\|_{L^2}$, then by the variational characterization of $W$ we must have $u=W$ (modulo symmetries).  In particular, $u$ is global.  Thus it suffices to treat the case of strict inequality in the statement above.

Now let $u:I\times\R^4\to\C$ be as in the statement of the proposition and suppose towards a contradiction that $\sup I<\infty$. (The case that $\inf I$ is finite is similar.) By Lemma~\ref{alcompact}, we have that $u$ is almost periodic on $[0,\sup I)$. 

We first claim that
\begin{align}\label{f2.13}
\liminf\limits_{t\nearrow\sup I} N(t) =\infty.
\end{align}
Indeed, if \eqref{f2.13} fails, then we can choose $t_n \to \sup{I}$ such that $\lim_n N(t_n)<\infty$, and define the rescaled functions $v_n: I_n\times\R^4\to \C$ by
\begin{equation*}
v_n(t,x):=  \tfrac{1}{N(t_n)}u\bigl( t_n + \tfrac{t}{N(t_n)^2}, \tfrac{x}{N(t_n)}+x(t_n)\bigr),
\end{equation*}
where 
\[
0\in I_n:=\bigl\{t\in \R:\ t_n+\tfrac{t}{N(t_n)^{2}}\in I  \bigr\}.
\]
Then each $v_n(t)$ is the solution to \eqref{1.1} and $\left\{v_n(0)\right\}_n$ is precompact in $\dot{H}^1(\R^4)$.  Passing to a subsequence, we may assume that $v_n(0)$ converges in $\dot{H}^1(\R^4)$ to some function $v_0(x)$.   By Lemma~\ref{varlemma}, we have $E(v_0)=E(W)$ and $\|\nabla v_0\|_{L^2}\leq \|\nabla W\|_{L^2}$.

We let $v$ be the maximal-lifespan solution to \eqref{1.1} with $v|_{t=0}=v_0$. Using the local theory for \eqref{1.1}, we may obtain that $\limsup_n S_{J}(v_n)<\infty$ for some compact interval $J\ni 0$. In particular, we have $\limsup_n S_{J_n}(u)<\infty$, where 
\[
J_n:=\bigl\{t_n+\tfrac{t}{N(t_n)^{2}}:\ t\in J\bigr\}.
\]
However, as $t_n\to\sup I$ and $\lim_n N(t_n)<\infty$, a further application of the local theory implies that $u$ exists beyond time $\sup I$, a contradiction.  This proves \eqref{f2.13}.

Next we claim that \eqref{f2.13} implies
\begin{align}\label{f2.14}
%\limsup_{t\to\sup I} \int_{|x|\leq R} |u(t,x)|^2\ \,dx=0\qtq{for any}R>0.
\limsup_{t\to \sup I}\|u(t)\chi_R\|_{L^2} = 0 \qtq{for any}R>0,
\end{align}
where $\chi_R$ is a smooth cutoff to $\{|x|\leq R\}$. Indeed, given any $R>0$ and $0<\eta\ll1$ we can use H\"older's inequality, Sobolev embedding, and Lemma~\ref{GlobalS} (which yields uniform $\dot H^1$-bounds) to estimate
\begin{align}\nonumber
\| u(t)\chi_R\|_{L^2} &\leq \|u(t)\|_{L^2(|x-x(t)|\leq \eta R)} + \|u(t)\chi_R\|_{L^2(|x-x(t)|>\eta R)} \\\nonumber
&\lesssim \eta R \|u(t)\|_{L^4}+ R\|u(t)\|_{L^4(|x-x(t)|>\eta R)} \\
& \lesssim \eta R + R\|u(t)\|_{L^4(|x-x(t)|>\eta R)}.
\end{align}
Now, given $\eps>0$ we choose $\eta$ small enough that $\eta R<\eps$.  We then use almost periodicity and the fact that $N(t)\to\infty$ as $t\to\sup I$ to see that the the second term vanishes as $t\to\sup I$. 

Now observe that using \eqref{1.1} and integrating by parts, we have
\[
\tfrac{d}{dt} \|u(t)\chi_R\|_{L^2}^2 \lesssim R^{-1} \|\nabla u(t)\|_{L^2} \| u(t)\|_{L^4} \|\tilde\chi_R\|_{L^4}
\]
for some function $\tilde \chi_R$ that is also localized to $|x|\leq R$.  In particular, the quantity above is uniformly bounded in $t$ and $R$, so that for any $t_1<t_2<\sup I$ we have from the Fundamental Theorem of Calculus that
\[
\|u(t_1)\chi_R\|_{L^2}^2 \leq \|u(t_2)\chi_R\|_{L^2}^2 + C|t_2-t_1|
\]
for some universal $C>0$. In particular, using \eqref{f2.14}, we obtain
\[
\|u(t)\chi_R\|_{L^2}^2 \leq C|\sup I - t| \qtq{for any} t<\sup I\qtq{and for any}R>0.
\]
In particular, this implies that $u$ is an $L^2$ solution obeying
\[
\|u(t)\|_{L^2}^2 \leq C|\sup I -t| \qtq{for any} t<\sup I.
\]
As mass is conserved for $L^2$ solutions and $\sup I$ is assumed to be finite, this inequality forces $u$ to be identically zero. However, this contradicts that $E(u)=E(W)>0$. Thus we conclude that $\sup I = \infty$, as desired. \end{proof}

%%%%%%%%%%%%%%%%%%%%%%%%%%%%%%%%%%%%%%
%%%%%%%%%%%%%%%%%%%%%%%%%%%%%%%%%%%%%%
%%%%%%%%%%%%%%%%%%%%%%%%%%%%%%%%%%%%%%
%%%%%%%%%%%%%%%%%%%%%%%%%%%%%%%%%%%%%%
%%%%%%%%%%%%%%%%%%%%%%%%%%%%%%%%%%%%%%
%%%%%%%%%%%%%%%%%%%%%%%%%%%%%%%%%%%%%%
\section{A sequential convergence result}

The goal of this section is to establish Proposition~\ref{sequence}, which states that a subcritical threshold solution that fails to scatter must converge to the ground state (in the sense that $\delta(u)\to 0$) along some sequence of times.

\begin{proposition}\label{sequence}
Suppose $u$ is a solution to \eqref{1.1} satisfying
\[
E(u) = E(W) \qtq{and} \|\nabla u_0\|_{L^2}<\|\nabla W\|_{L^2}.
\]
In particular, by Proposition~\ref{T:no ftb}, $u$ is global. If $u$ blows up forward or backward in time, then there exists a sequence $\{t_n\}\subset\R$ such that $\delta(u(t_n))\to 0$.  
\end{proposition}

\begin{proof} Without loss of generality, we assume that $u$ blows up forward in time.  Our first step will be to argue for the existence of a global solution that blows up in both time directions (and hence is almost periodic on all of $\R$).

We let $\tau_n\to\infty$ be an increasing sequence.  Applying Lemma~\ref{alcompact}, we may write 
\[
u(\tau_n) = g_n \phi + w_n
\]
for some group elements $g_n$ (corresponding to scales $\lambda_n$ and translation parameters $x_n$) and functions $w_n$ satisfying $\|w_n\|_{\dot H^1}\to 0$.

We let $v$ be the solution to \eqref{1.1} with $v|_{t=0}=\phi$.  By the $\dot H^1$ convergence, we have $E(v)=E(W)$ and $\|v(0)\|_{\dot H^1}\leq \| W\|_{\dot H^1}$.  In particular, by Proposition~\ref{T:no ftb} the solution $v$ is global.  On the other hand, as
\[
S_{(-\infty,\tau_n]}(u) \to \infty \qtq{and} S_{[\tau_n,\infty)}(u)\to\infty,
\]
we must have $S_{(-\infty,0]}(v)=S_{[0,\infty)}(v)=\infty$.  Thus by Lemma~\ref{alcompact}, $v$ is almost periodic on $\R$.  

We now assert the following: there exist $\{s_m\}\subset\R$ and group elements $\{g_m\}$ such that 
\begin{equation}\label{th5v26}
\delta(v(s_m))=\delta(g_m v(s_m))\to 0 \qtq{as}m\to\infty. 
\end{equation}
Assuming \eqref{th5v26} for the moment, let us complete the proof of Proposition~\ref{sequence}:

First, passing to a subsequence, we may assume that
\begin{equation}\nonumber%\label{qzaaa}
\delta(g_mv(s_m ))\leq 2^{-m }.
\end{equation}
To keep the formulas in the margins below, let us introduce the notation
\[
\tilde u_{m,n}(x) = \lambda_n u(\tau_n+\lambda_n^2 s_m,\lambda_nx + x_n).
\]
By construction,
\[
\lambda_n u(\tau_n, \lambda_n x + x_n) \to \phi(x) \qtq{in}\dot H^1.
\]
Thus, for fixed $m$, we have by the stability result (Lemma~\ref{pertu}) that for $n=n(m)$ sufficiently large
\begin{align*}
\|\tilde u_{m,n}(x) - v(s_m,x)\|_{\dot H^1} &\leq C(s_m) \|\lambda_n u(\tau_n,\lambda_n x + x_n)-\phi(x)\|_{\dot H^1} \\
& \leq C(s_m) o_n(1) \qtq{as}n\to\infty. 
\end{align*}
Thus by the triangle inequality (and uniform $\dot H^1$-boundedness of $v$ and $u$)
\begin{align*}
\delta(u(\tau_n+\lambda_n^2 s_m)) & = \delta(\tilde u_{m,n}) \leq |\delta(v(s_m))| + \bigl| \| v(s_m)\|_{\dot H^1}^2 - \|\tilde u_{m,n}\|_{\dot H^1}^2\bigr| \\
& \lesssim |\delta(v(s_m))| + C(s_m)o_n(1) \\
& \lesssim 2^{-m} + C(s_m)o_n(1).
\end{align*}
Thus for each $m$, there exists $n=n(m)$ such that 
\[
\delta(u(\tau_{n(m)}+\lambda_{n(m)}^2 s_m))\lesssim 2^{-m},
\]
so that Proposition~\ref{sequence} holds with $t_m:=\tau_{n(m)}+\lambda_{n(m)}^2 s_m$.

It therefore remains to prove \eqref{th5v26}.  

To this end, we will first argue for the existence of another almost periodic solution, which has frequency scale bounded from below.  The construction, which utilizes the notion of the normalization of an almost periodic solution, is standard by now (see e.g. \cite[Section~5.4]{KV1}).  Thus we will omit the details and instead refer the reader to \cite{KV1} for a clear and thorough exposition.

In particular, we may construct a solution $w:\R\times\R^4\to\C$ to \eqref{1.1} such that
\[
w(0) = \lim_{n\to\infty} \tilde g_n v(r_n)
\]
for some group elements $\tilde g_n$ and some sequence $\{r_n\}\subset\R$.  This solution satisfies $E(w)=E(W)$ and $\|\nabla w(0)\|_{L^2}\leq \|\nabla W\|_{L^2}$ and is almost periodic, with some spatial center $x(t)$ and frequency scale $N(t)$ satisfying $\inf_{t\in\R}N(t)\geq 1$.  
 
Arguing as we did above, we now observe that to prove \eqref{th5v26} it suffices to show that
\begin{equation}\label{th5v262}
\qtq{there exists}\{s_m\}\subset\R\qtq{such that}\delta(w(s_m))\to 0\quad \text{as}\;m\to\infty.
\end{equation}
The proof of \eqref{th5v262} follows from analysis of the solution $w$ very much in the spirit of Dodson's work \cite{dodson1}.  In fact, the arguments of \cite{dodson1} can essentially be applied directly in our setting to obtain this result, as we will now explain.

We begin by reviewing the result and strategy of \cite{dodson1} in fairly general terms.  We provide additional technical details in Section~\ref{S:Dodson} below. 

%\textcolor{red}{To maintain consistency with the notation in \cite{dodson1}, should we replace the interval $\R$ with $[0,\infty)$ in the following context?} {\color{blue} Looking back at \cite{dodson1}, it seems he works mostly on all of $\R$... so I think it's OK to leave it as $\R$...} 

As discussed in the introduction, \cite{dodson1} establishes scattering for solutions $w$ to \eqref{1.1} obeying $E(w)<E(W)$ and $\|\nabla w(0)\|_{L^2}<\|\nabla W\|_{L^2}$.  The proof is by contradiction.  Under the assumption that the result is false, Dodson first deduces the existence of a global almost periodic solution $w$ satisfying $E(w)<E(W)$ and $\|\nabla w(0)\|_{L^2}<\|\nabla W\|_{L^2}$, with frequency scale function satisfying $\inf N(t)\geq 1$.  Dodson then considers the following two alternatives, described in terms of the behavior of $N(t)$:
\begin{itemize}
\item[(i)] $K:=\int_\R N(t)^{-2}\,dt<\infty$, the \emph{frequency cascade} scenario, and
\item[(ii)] $K:=\int_\R N(t)^{-2}\,dt=\infty$, the \emph{quasi-soliton} scenario. 
\end{itemize}

To complete the proof, Dodson must prove that neither scenario can occur.  Essential to the analysis are Dodson's \emph{long-time Strichartz estimates}, which provide control over space-time norms of the frequency components of the solution in terms of the quantity $K$.  These estimates only require almost periodicity and are not related to any scattering threshold; indeed, they are completely insensitive to the sign of the nonlinearity.  Another technical ingredient is the establishment of additional decay for almost periodic solutions with $N(t)\geq 1$, e.g. the estimate $w\in L_t^\infty L_x^3$.  This estimate is also insensitive to the sign of the nonlinearity and hence has no connection to the scattering threshold.  

Using the long-time Strichartz estimates and conservation of mass, Dodson precludes scenario (i).  Specifically, one can show that almost periodic solutions escaping to high frequencies must have zero mass, which yields a contradiction.  As none of the tools involved in the analysis have any connection to the scattering threshold, we can therefore deduce that the particular solution $w$ we have constructed above must fall into the quasi-soliton scenario.

Dodson's final step is to preclude the existence of quasi-solitons, and it is here that he uses the sub-threshold assumption in an essential way.  He combines the additional decay, the long-time Strichartz estimates, a double Duhamel argument yielding a local mass estimate with logarithmic loss, and the subthreshold assumption to establish an interaction Morawetz estimate that ultimately precludes the possibility of a quasi-soliton.  In particular, {his arguments establish} the following result, which we explain in Section~\ref{S:Dodson} in some detail:

\begin{proposition}[No subcritical quasi-solitons, \cite{dodson1}]\label{P:nscqs}  Suppose $u:\R\times\R^4\to\C$ is an almost periodic solution to \eqref{1.1} satisfying
\[
\sup_{t\in\R} \|u(t)\|_{\dot H_x^1}<\|W\|_{\dot H^1},\quad \inf_{t\in\R}N(t)\geq 1,\qtq{and} \int_\R N(t)^{-2}\,dt = \infty.
\]
Then $u\equiv 0$. 
\end{proposition}

Using Proposition~\ref{P:nscqs}, we can complete the proof of \eqref{th5v262} and hence of Proposition~\ref{sequence}.  Indeed, we have already described above that the solution $w$ we have extracted must fall into the quasi-soliton scenario.  As $w$ is nonzero, Lemma~\ref{GlobalS} and Proposition~\ref{P:nscqs} therefore imply that
\[
\sup_{t\in\R}\|w(t)\|_{\dot H_x^1} = \|W\|_{\dot H_x^1}. 
\]
In particular, there exists a sequence $\{s_m\}\subset\R$ such that $\delta(w(s_m))\to 0$. \end{proof}

%%%%%%%%%%%%%
%%%%%%%%%%%%%
\subsection{Dodson's analysis of almost periodic solutions}\label{S:Dodson} In this subsection, we follow the presentation in \cite{dodson1} and sketch some of the technical details related to Proposition~\ref{P:nscqs}.  In particular, this proposition follows from the interaction Morawetz of \cite{dodson1}, as we now describe. 

In what follows, we assume $u:\R\times\R^4\to\C$ is a \emph{nonzero} solution to \eqref{1.1} satisfying the hypotheses of Proposition~\ref{P:nscqs} and seek a contradiction.  We denote the compactness parameters of $u$ by $x(t)$ and $N(t)$.  We fix an interval $I=[t_1,t_2]$ and let $K=K_I=\int_I N(t)^{-2}\,dt$.  As we are in the quasi-soliton scenario, we may assume $K$ is arbitrarily large by taking $I$ sufficiently large inside $\R$.

We let $\psi\in C_c^\infty(\R^4)$ be a radial function satisfying
\[
\psi(x) = \begin{cases} 1 & |x|\leq 1 \\ 0 & |x|\geq2.\end{cases}
\]
We then let $J\gg1$ (later chosen such that $e^J = K^{\frac{1}{12}}$) and define
\[
\phi(x-y) = \tfrac{1}{J}\int_1^{e^J} \int_{\R^4} \psi^2(\tfrac{x}{R}-s)\psi^2(\tfrac{y}{R}-s)\,ds\,\tfrac{dR}{R}. 
\]

The interaction Morawetz quantity is then defined to be
\[
M(t) = \iint |u(y)|^2 \phi((x-y)n(t))(x-y)\cdot \Im[\bar u \nabla u](x)\,dx
\]
for a certain function $n(t)$ that is related to the frequency scale function of $u$ (denoted by $N_m(t)$ in the reference \cite{dodson1}).  

The choice of $n(t)$ is one of the many subtle (and brilliant) points of Dodson's analysis. Given that the physical scale of $u$ is $N(t)^{-1}$, a natural choice would be to take $n(t)=N(t)$.  However, when one differentiates $M(t)$ with respect to time, one will then encounter a term containing $N'(t)$, over which one does not have good control.  Dodson's approach is to choose $n(t)$ to be the output of a `smoothing algorithm', which takes as input a function closely related to $N(t)$ and returns progressively smoother/monotone functions that retain certain essential features.

In particular, given an interval $I$, Dodson first chooses
\[
n_0(t) := \|P_{>K^{-\frac14}} u(t)\|_{L_x^3}^{-3}, 
\]
which satisfies $1\lesssim n_0(t)\lesssim N(t)$. Dodson then defines an algorithm (see \cite[Definition~6.1]{dodson1}) that produces functions $n_m$ satisfying the following for $t\in I$:
\begin{itemize}
\item $n_0(t)\lesssim n_m(t) \lesssim 2^m n_0(t),$
\item $n_m'(t) = 0 \qtq{or} n_m(t) = n_1(t),$
\item $|n_m'(t)|\lesssim |n_m(t)|^3,$
\item $\int n_m(t)^{-2}\,dt \lesssim K,$
\item $\int |n_m'(t)| [n_m(t)]^{-5}\,dt \lesssim 1+2^{-4m}K.$
\end{itemize}

By choosing $m$ large enough, one can obtain sufficiently good control over the term in the Morawetz estimate containing $n_m'(t)$ (see \eqref{Dod-6} below). On the other hand, $n_m(t)$ grows (pointwise in $t$) as $m$ increases, while the weight appearing in $M(t)$ is only nonzero for $|x-y|\lesssim \tfrac{e^J}{n_m(t)}$. Thus one must choose the various parameters carefully in order to balance all of these considerations, so that for example $\tfrac{e^J}{n_m(t)} \gtrsim \tfrac{K^{\eps}}{N(t)}$ for some $\eps>0$. This is ultimately done by choosing $2^m$ to be some fractional power of $K$.  Similar techniques were also employed in \cite{dodson2} and are explained in some detail there (see e.g. \cite[Section~5.3]{dodson2}).

At the most basic level, the interaction Morawetz is then based on using the Fundamental Theorem of Calculus in the form
\begin{equation}\label{FTCID}
\int_{t_1}^{t_2} \tfrac{dM}{dt}\,dt = M(t_2)-M(t_1).
\end{equation}
One seeks a suitable upper bound on $\sup_{t\in I}|M(t)|$ and a suitable \emph{lower} bound on $\tfrac{dM}{dt}$ in order to obtain a meaningful estimate. 

By construction (and the fact that the function $n(t)$ is bounded below) one readily obtains the upper bound
\[
\sup_{t\in I} |M(t)| \lesssim e^{4J}. 
\]

Using subscripts to denote partial derivatives, explicit calculation using \eqref{1.1} and integration by parts shows that
\begin{align}
\tfrac{dM}{dt} & = 2\iint |u(y)|^2 \phi((x-y)n(t))[|\nabla u(x)|^2 - |u(x)|^4]\,dx\,dy \label{Dod-1} \\
& \ - 2\iint \Im[\bar u u_j](y) \phi((x-y)n(t))\Im[\bar u u_j](x)\,dx\,dy \label{Dod-2} \\
& \ + 2\iint |u(y)|^2\phi_k((x-y)n(t)) (x-y)_j[\Re[\bar u_j u_k] - \tfrac{\delta_{jk}}4|u|^4](x)\,dx\,dy \label{Dod-3} \\
& \  - 2\iint \Im[\bar u u_k](y)\phi_k((x-y)n(t))(x-y)_j \Im[\bar u u_j](x)\,dx\,dy \label{Dod-4} \\
& \  + \tfrac12\iint |u(y)|^2 \Delta\phi_j((x-y)n(t))(x-y)_j |u(x)|^2\,dx\,dy \label{Dod-5}\\
& \  + {\iint |u(y)|^2\phi_k((x-y)n(t))(x-y)_j(x-y)_k n'(t) \Im[\bar u u_j](x)\,dx\,dy}.\label{Dod-6}
\end{align}
Note that we often suppress explicit dependence of functions on time in order to keep formulas within the margins.  

The terms \eqref{Dod-1} and \eqref{Dod-2} are used to obtain the desired \emph{lower} bound, as we now describe.  First, rearrange these terms as follows:
\begin{align}
&\eqref{Dod-1}+\eqref{Dod-2} \nonumber \\
&\ = 2\iint \phi((x-y)n(y))[|u(y)|^2 |\nabla u(x)|^2 - \Im[\bar u u_j](y)\Im[\bar u u_j](x)]\,dx\,dy \label{Dod-invariant} \\
& \quad - 2\iint \phi((x-y)n(y))|u(y)|^2 | u(x)|^4\,dx\,dy. \label{Dod-leftovers}
\end{align}
Recalling the definition of $\phi$, one can verify directly that the expression
\[
\iint \psi^2(\tfrac{ny}{R}-s)\psi^2(\tfrac{nx}{R}-s)[|u(y)|^2 |\nabla u(x)|^2 - \Im[\bar uu_j](y)\Im[\bar u u_j](x)]\,dx\,dy
\]
(which is integrated appropriately in $s$ and $R$ to obtain \eqref{Dod-invariant}) is invariant under the replacement $u\mapsto e^{-ix\cdot\xi}u$ (for fixed $R$, $t$, and $s$).  We may therefore replace $u$ in this expression by $e^{-ix\cdot\xi}u$ for any choice of $\xi$.  In particular, we choose $\xi=\xi(t,s,R)$ such that (writing $u^\xi = e^{-ix\cdot \xi}u$ here and below) 
\[
\int_{\R^4}\psi^2(\tfrac{n(t)}{R}x-s)\Im[\bar u^\xi \nabla (u^\xi)](t,x)\,dx = 0. 
\]
Making this choice, we therefore obtain
\begin{align}
&\eqref{Dod-1}+\eqref{Dod-2}\nonumber \\
&\ = \tfrac{2}{J}\int_\Omega \psi^2(\tfrac{ny}{R}-s)\psi^2(\tfrac{nx}{R}-s)|u(y)|^2\bigl(|\nabla[u^\xi](x)|^2 - |u(x)|^4\bigr)\,d\Omega
\end{align}
where the integral is over $(x,y,s)\in\R^4\times\R^4\times\R^4$ and $R\in[1,e^J]$ and we write $d\Omega=\,dx\,dy\,ds\,\tfrac{dR}{R}$ to keep the formulas within the margins. We now apply the following integration by parts identity in the $x$ variable: for real-valued $\psi$, 
\[
\int |\nabla(\psi f)|^2 \,dx = \int \psi^2 |\nabla f|^2 \,dx- \int \psi \Delta\psi |f|^2\,dx.
\]
After a bit of rearranging, this yields
\begin{align}
&\eqref{Dod-1}+\eqref{Dod-2}\nonumber \\
&\ = \tfrac{2}{J} \int_\Omega \psi^2(\tfrac{ny}{R}-s)|u(y)|^2\bigl[ |\nabla [\psi(\tfrac{nx}{R}-s)u^\xi](x)|^2-\psi^2(\tfrac{nx}{R}-s)|u(x)|^4\bigr]\,d\Omega\label{Dod-LB-cont1} \\
& \ \ + \tfrac{2}{J}\int_\Omega \psi^2(\tfrac{ny}{R}-s)\psi(\tfrac{nx}{R}-s)\Delta[\psi(\tfrac{nx}{R}-s)]|u(y)|^2|u(x)|^2\,d\Omega.\label{Dod-LB-cont2}
\end{align}

Now we are finally in a position to utilize the subcritical assumption and sharp Sobolev embedding to exhibit a lower bound for the term \eqref{Dod-LB-cont1}. By assumption, there exists $\delta>0$ such that
\[
\|\nabla u(t)\|_{L_x^2}\leq (1-\delta)\|\nabla W\|_{L^2} \qtq{for all}t\in\R. 
\]
Applying the sharp Sobolev embedding estimate (Lemma~\ref{GNsharpz}), this also guarantees that
\[
\| u(t)\|_{L_x^4} \leq (1-\delta)\|W\|_{L^4}\qtq{for all}t\in\R. 
\]
Thus, focusing attention on the $\,dx$ portion of \eqref{Dod-LB-cont1} (and noting $|u|\equiv |u^\xi|$), we observe that by H\"older's inequality and sharp Sobolev embedding (Lemma~\ref{GNsharpz}),
\begin{align*}
\int &|\nabla[\psi(\tfrac{nx}{R}-s)u^\xi]u(x)|^2 - |\psi(\tfrac{nx}{R}-s) u^\xi|^2 |u|^2\,dx \\
& \geq \|\nabla[\psi(\tfrac{nx}{R}-s)u^\xi]\|_{L_x^2}^2 - \|\psi(\tfrac{nx}{R}-s)u^\xi\|_{L_x^4}^2 \|u\|_{L_x^4}^2 \\ 
& \geq [C_4^{-2}-\|u\|_{L_x^4}^2]\|\psi(\tfrac{nx}{R}-s)u^\xi\|_{L_x^4}^2 \\
& \geq [\|W\|_{L_x^4}^2 - \|u\|_{L_x^4}^2]\|\psi(\tfrac{nx}{R}-s)u\|_{L_x^4}^2 \\
& \gtrsim \delta \|W\|_{L_x^4}^2\|\psi(\tfrac{nx}{R}-s)u\|_{L_x^4}^2 \\
& \gtrsim \delta \|u\|_{L_x^4}^2\|\psi(\tfrac{nx}{R}-s)u\|_{L_x^4}^4 \\
& \gtrsim \delta \|\psi(\tfrac{nx}{R}-s)u\|_{L_x^4}^4
\end{align*}
uniformly in $t$.  It follows that
\begin{align*}
\eqref{Dod-LB-cont1} & \gtrsim \tfrac{\delta}{J} \int_\Omega \psi^2(\tfrac{ny}{R}-s)\psi^4(\tfrac{nx}{R}-s)|u(y)|^2|u(x)|^4\,d\Omega. 
\end{align*}
By construction, one can further obtain the lower bound
\[
\tfrac{1}{J}\int_1^{e^J} \int \psi^2(\tfrac{ny}{R}-s)^2\psi^4(\tfrac{nx}{R}-s)\,ds\,\tfrac{dR}{R} \gtrsim_c \psi(\tfrac{n(x-y)}{e^{cJ}})
\]
for any $0<c<1$.  For the sake of concreteness, Dodson chooses $c=\tfrac{11}{12}$.  We follow this choice as well, but we will continue to write $c$ to keep the formulas looking neat. Thus one can use almost periodicity to obtain the following lower bound:
\begin{equation}\label{the-lower-bound!}
\begin{aligned}
\eqref{Dod-LB-cont1}& \gtrsim  \delta \iint_{|x-y|\leq \frac{e^{cJ}}{n(t)}}|u(t,x)|^4 |u(t,y)|^2 \,dx\,dy \\
& \gtrsim \delta \int_{|x-x(t)|\leq \frac{e^{cJ}}{2n(t)}}|u(t,x)|^4\,dx \cdot \int_{|y-x(t)|\leq \frac{e^{cJ}}{2n(t)}} |u(t,y)|^2\,dy \\
& \gtrsim \delta \int_{|y-x(t)|\leq \frac{e^{cJ}}{2n(t)}} |u(t,y)|^2\,dy,
\end{aligned}
\end{equation}
where we have used that $\frac{e^{cJ}}{n(t)}\gg \tfrac{1}{N(t)}$ uniformly on $I$ (cf. the discussion of $n(t)$ above).  One therefore has the lower bound
\[
\eqref{Dod-1}+\eqref{Dod-2} \gtrsim \delta \int_{|y-x(t)|\leq \frac{e^{cJ}}{2n(t)}} |u(t,y)|^2\,dy - |\eqref{Dod-LB-cont2}|,
\]
with \eqref{Dod-LB-cont2} now regarded as an error term (that is, a term for which we will obtain a suitable \emph{upper} bound). 

We note that it is \emph{only} this part of the argument (i.e. obtaining a lower bound of the form \eqref{the-lower-bound!}) that uses the subcritical assumption on the $\dot H^1$-norm of the solution (in connection with sharp Sobolev embedding).  Moreover, the argument does \emph{not} directly use any assumption about the size of the total energy. This is the essential reason why we can apply Dodson's argument {directly} in the present paper. 

Having obtained a lower bound for \eqref{Dod-LB-cont1}, we next consider the error terms \eqref{Dod-LB-cont2} and \eqref{Dod-3}--\eqref{Dod-6} in $\tfrac{dM}{dt}$.  For \eqref{Dod-LB-cont2}, one uses properties of the function $\psi$ to obtain
\[
|\eqref{Dod-LB-cont2}| \lesssim \iint_{|x-y|\leq\frac{e^J}{n(t)}} |u(y)|^2 |u(x)|^2 |x-y|^{-2}\,dx\,dy,
\]
which turns out to be similar to \eqref{Dod-5}. Estimating the terms \eqref{Dod-3}--\eqref{Dod-6} then comprises a great deal of hard work in \cite{dodson1}.  It is here that Dodson employs his long-time Strichartz estimates, which (as mentioned above) depend only on almost periodicity (and are not even sensitive to the sign of the nonlinearity).  

In the end, Dodson shows that (integrating in time as in \eqref{FTCID} and utilizing the lower bound \eqref{the-lower-bound!}) one can obtain an estimate of the form
\begin{equation}\label{dodson-final1}
\begin{aligned}
\delta&\int_I \int_{|x-x(t)|\leq \frac{e^{cJ}}{2n(t)}} |u(t,x)|^2 \,dx\,dt \\
& \lesssim \tfrac{1}{J}\int_I \int_{|x-x(t)|\leq \frac{8e^J}{n(t)}} |u(t,x)|^2\,dx + (\eta+\tfrac{1}{Jc(\eta)}) K+ \tfrac{2^{-4m}e^{3J}}{J} + K^{\frac12}\tfrac{e^{5J}}{J} + e^{4J},
\end{aligned}
\end{equation}
where $0<\eta\ll1$ and each of the quantities on the right-hand side arises from estimating either $|M(t)|$ or some error term.  The expression with the factor $2^{-4m}$ arises from \eqref{Dod-6} (the term containing $n'(t)$).  The first term on the right-hand side is a \emph{localized mass} term arising from the terms \eqref{Dod-3}--\eqref{Dod-4}, and its estimation is another subtle point of Dodson's analysis.  

On the one hand, this term comes with the small factor $J^{-1}$, suggesting that it may be treated perturbatively.  However, the only direct estimate available for a localized term such as this one suffers a logarithmic loss that would cancel this factor (see \cite[Lemma~5.1]{dodson1}).  In particular, the estimate (which only relies on almost periodicity) implies that 
\begin{equation}\label{lemma5.1}
 \int_I \int_{|x-x(t)|\leq \frac{R}{N(t)}}|u(t,x)|^2\,dx\,dt \lesssim K\log(R)\qtq{for}1\ll R\lesssim K^{\frac15}.
\end{equation}
In the present setting, an application of this estimate would lead to an estimate of $JK$ for the localized mass term (cf. $J\sim\log K$), and the logarithmic loss would cancel the small factor $J^{-1}$.  In particular, this approach would not suffice to close the argument. 

On the other hand, the localized mass term on the right-hand side of \eqref{dodson-final1} strongly resembles the quantity that we have controlled on the \emph{left-hand side} of \eqref{dodson-final1}.  The only difference is that the radius for the term on the right-hand side is a larger multiple of $\tfrac{1}{n(t)}$, namely, $e^J$ instead of $e^{cJ}$.  Thus, we can apply the \emph{same} interaction Morawetz estimate (or repeat the same arguments) to obtain the following bound for this localized mass term:
\begin{align*}
\delta\int_I \int_{|x-x(t)|\leq \frac{8e^J}{n(t)}} |u(t,x)|^2\,dx\,dt &  \lesssim \tfrac{1}{J}\int_I \int_{|x-x(t)| \leq \frac{(8J)^{1/c}}{n(t)}} |u(t,x)|^2\,dx\,dt \\
& \quad + (\eta+\tfrac{1}{Jc(\eta)}) K + \tfrac{2^{-4m} e^{3J/c}}{J} + K^{\frac12} \tfrac{e^{5J/c}}{J} + e^{4J/c}. 
\end{align*}

We can now insert \emph{this} estimate back into \eqref{dodson-final1}.  At this point it is also convenient to specify parameters and simplify the right-hand side.  Choosing $J$ such that $e^{12J}=K$ and $m$ such that $2^{4m}=e^{\frac{10J}{3}}$, the estimate \eqref{dodson-final1} then reduces to
\begin{equation}\label{dodson-final2}
\begin{aligned}
\delta \int_I \int_{|x-x(t)|\leq \frac{e^{cJ}}{n(t)}}|u(t,x)|^2\,dx\,dt & \lesssim \tfrac{1}{\delta J^2}\int_I \int_{|x-x(t)|\leq \frac{(8J)^{1/c}}{n(t)}} |u(t,x)|^2\,dx \\
& \quad + (\eta+\tfrac{1}{c(\eta)\log K})K.
\end{aligned}
\end{equation}
Of course, at this point we must \emph{still} contend with a localized mass term on the right-hand side.  The point is that by passing through the interaction Morawetz estimate one time, we have gained an additional factor of $J^{-1}$ in this term.  Thus it is now sufficient to apply the lossy estimate \eqref{lemma5.1} to this local mass term, as the logarithmic loss in this estimate will cancel only \emph{one} of the factors of $J^{-1}$.  In particular, we have
\[
 \tfrac{1}{J^2}\int_I \int_{|x-x(t)|\leq \frac{(8J)^{1/c}}{n(t)}}|u(t,x)|^2\,dx \lesssim \tfrac{K}{\log K},
\]
so that from \eqref{dodson-final2} one finally arrives at the estimate 
\begin{equation}\label{dodson-final3}
\delta^2 \int_I \int_{|x-x(t)|\leq \frac{e^{cJ}}{n(t)}}|u(t,x)|^2\,dx\,dt\lesssim (\eta+\tfrac{1}{c(\eta)\log K})K.
\end{equation}
However, as one can use almost periodicity to show that
\[
\int_I \int_{|x-x(t)|\leq \frac{C}{N(t)}}|u(t,x)|^2\,dx\,dt \gtrsim \int_I N(t)^{-2}\,dt =K,
\]
the inequality \eqref{dodson-final3} leads a contradiction provided $K$ is sufficiently large.  This completes the proof.  For the complete analysis, we refer the reader again to \cite{dodson1}.

%%%%%%%%%%%%%%%%%%%%%%%%%%%%%%%%%%%%%%%%%%%%
%%%%%%%%%%%%%%%%%%%%%%%%%%%%%%%%%%%%%%%%%%%%
%%%%%%%%%%%%%%%%%%%%%%%%%%%%%%%%%%%%%%%%%%%%
\section{Modulation analysis}\label{S:modulation}

In this section, we collect some results related to modulation analysis for \eqref{1.1}.  This refers to a description of solutions at times when they are near the orbit of the ground state, as measured by the functional
\[
\delta(f):=\|W\|_{\dot{H}^1}^2-\|f\|_{\dot{H}_x^1}^2.
\]
The description involves several modulation parameters, including a phase, scale, and spatial center.

We can import much of what we need from \cite{duyckaerts1}, with minor modifications taken to address the spatial center (which only appears in the non-radial setting).  In what follows, orthogonality is taken with respect to the $\dot H^1$ inner product. 

Given $v\in\dot{H}^1(\R^4)$, we define
\begin{equation*}
v_{\left[\theta_0,\lambda_0,x_0\right]}(x):=e^{i\theta_0} \tfrac{1}{\lambda_0}v(\tfrac{x-x_0}{\lambda_0})
\end{equation*}

Adapting the arguments used to prove \cite[Lemma~3.6]{duyckaerts1}, we can obtain the following lemma.  The basic idea is that the variational characterization of $W$ provides an initial decomposition around the ground state, which can then be modified slightly to impose some desired orthogonality conditions (more on this below).  This final step relies on the implicit function theorem, which in turn relies on direct calculation to verify the nondegeneracy of the appropriate partial derivatives.

\begin{lemma}\label{Mod1} There exists $\delta_0>0$ for any $f\in\dot H^1$ with $E(f)=E(W)$ and $|\delta(f)|<\delta_0$, there exist unique $(\theta,\lambda,\tilde x)\in\mathbb{R}/2\pi\mathbb{Z}\times(0,\infty)\times\R^4$ such that
\[
f_{[\theta,\lambda,\tilde{x}]}  \perp iW,W_1,\partial_j W, \quad j\in\{1,2,3,4\},
\]
where 
\[
W_1:=W+x\cdot \nabla W.
\]
Moreover, the mapping $f \mapsto (\theta,\lambda,\tilde{x})$ is $C^1$.\end{lemma}

We now apply Lemma~\ref{Mod1} to a solution to \eqref{1.1}.  We suppose $u:I\times\R^4\to\C$ is a solution to \eqref{1.1} satisfying 
\begin{equation}\label{u-subcritical}
E(u)=E(W),\quad\|\nabla u(t)\|_{L^2}<\|\nabla W\|_{L^2},\qtq{and} \delta(u(t))<\delta_0 \qtq{on}I.
\end{equation}
We further define 
\begin{equation}\label{defAorth}
\mathcal{A}:=\{W, iW, W_1, \partial_{1} W, \partial_{2} W, \partial_{3} W, \partial_{4} W\}.
\end{equation}
We choose $(\theta(t),\lambda(t),x(t))$ corresponding to $u(t)$ as in Lemma~\ref{Mod1} and further define $\alpha(t)$ via
\[
1+\alpha(t)=\frac{1}{\|W\|_{\dot{H^1}}^2}(u_{[\theta(t),\lambda(t),x(t)]},W)_{\dot{H}^1}.
\]
We can then write
\begin{equation}\label{ortho}
u_{[\theta(t),\lambda(t),x(t)]}(t)=\bigl[1+\alpha(t)\bigr]W+\tilde{u}(t),\qtq{with} \tilde u(t)\in \mathcal{A}^\perp.
\end{equation}
Finally, we define $v(t)$ by 
\[ 
v(t):=\alpha(t)W+\tilde{u}(t)=u_{[\theta(t),\lambda(t),x(t)]}(t)-W.
\]

We would like to point out that the notation $x(t)$ has previously been used to denote the one of the compactness parameters for almost periodic solutions.  Throughout this section, we always write $x(t)$ to refer to a modulation parameter, and throughout the paper we will always specify whether our notation refers to the modulation parameter or the compactness parameter.

We will need various estimates on the modulation parameters, which we record in the following lemma.

\begin{lemma}\label{modula}
Suppose $u:I\times\R^4\to\C$ is a solution to \eqref{1.1} satisfying \eqref{u-subcritical}. If $\delta_0$ is sufficiently small, then the following estimates hold on $I$:
\begin{align}
|\alpha(t)|\sim \|v(t)\|_{\dot{H}^1(\R^4)}\sim \|\tilde{u}(t)\|_{\dot{H}^1(\R^4)} \sim \delta(u(t)), \label{equivq1} \\
\bigl|{x^{\prime}(t)}-\tfrac{\lambda^{\prime}(t)}{\lambda(t)}x(t)\bigr|+|\alpha^{\prime}(t)|+|\theta^{\prime}(t)|+\bigl|\tfrac{\lambda^{\prime}(t)}{\lambda(t)}\bigr|\lesssim\lambda(t)^2 \delta(u(t)).\label{Mod2}    
\end{align}
\end{lemma}

The proof of Lemma~\ref{modula} relies on a coercivity result for the quadratic form that arises when one expands the energy around the ground state.  In particular, we note that we may write 
\begin{equation}\label{energyex}
E(W+g)=E(W)+\mathcal{F}(g)+o(\|g\|^2_{\dot{H}^1}) \qtq{for any}g\in\dot H^1,
\end{equation}
where 
\begin{equation}\label{def-quad-form}
\mathcal{F}(g):=\tfrac{1}{2}\int_{\R^4}|\nabla g|^2 \,dx-\tfrac{1}{2}\int_{\R^4}W^2[(3\Re g)^2+(\Im g)^2] \, dx.
\end{equation}
One then has the following lower bound on the orthogonal complement of the set $\mathcal{A}$ defined in \eqref{defAorth}. This estimate was previously established in \cite[Appendix D]{Olivier} (see also \cite[Claim 3.5]{duyckaerts3}), so we omit the proof here.

\begin{lemma}[Coercivity, \cite{Olivier, duyckaerts3}]\label{coer}
There exists constant $c>0$ such that 
\begin{equation}\label{defF}
\mathcal{F}(g)\geq c\|g\|^2_{\dot{H}^1}\qtq{for all}g\in\mathcal{A}^\perp.
\end{equation}
\end{lemma}

With Lemma~\ref{coer} in hand, we turn to the proof of Lemma~\ref{modula}.

\begin{proof}[Proof of Lemma~\ref{modula}]  In the radial case (in particular, in the absence of $x(t)$), estimates like \eqref{equivq1} and \eqref{Mod2} are proven in detail in \cite[Section~7]{duyckaerts1}.  We adapt these arguments to the non-radial setting,  which involves the parameter $x(t)$ in \eqref{Mod2}.

We first prove \eqref{equivq1}. Note that $\alpha\in\R$ is chosen to guarantee that $(\tilde{u}, W)_{\dot{H}^1}=0$, so that
 \begin{equation}\label{equal1}
\|v\|_{\dot{H}^1}^2=\alpha^2\|W\|_{\dot{H}^1}^2+\|\tilde{u}\|^2_{\dot{H}^1}.
\end{equation}
Moreover, recalling $\tilde u \in\mathcal{A}^\perp$, Lemma~\ref{coer} implies,
\begin{equation}\label{tildeu}
\mathcal{F}(\tilde{u})\gtrsim\|\tilde{u}\|^2_{\dot{H}^1}.
\end{equation}

Now, using \eqref{energyex} and $E(u)=E(W)$, we have
\begin{equation*}
0 = \mathcal{F}(\alpha W+\tilde{u} )+o(\|v\|^2_{\dot{H}^1}),\qtq{i.e.} \F(\alpha W + \tilde u)=o(\|v\|_{\dot H^1}^2).%=\mathcal{F}(v)+ o(\|v\|^2_{\dot{H}^1}).
\end{equation*}
As $\Delta W=-W^3$ and $\tilde{u}\in\mathcal{A}^{\perp}$, we have
\begin{equation*}
\int_{\R^4}W^3\Re \tilde{u} \, dx=\int_{\R^4}W^3\Im \tilde{u} \, dx=0,
\end{equation*}
so that
\begin{equation*}
-|\mathcal{F}(W)|\alpha^2+\mathcal{F}(\tilde{u})=\mathcal{F}(\alpha W+\tilde{u})=o(\|v\|^2_{\dot{H}^1}).
\end{equation*}
Combining this equality with \eqref{equal1}, \eqref{tildeu}, and the fact that 
\[
\mathcal{F}(\tilde{u})\lesssim \|\tilde{u}\|^2_{\dot{H}^1},
\]
we can derive that
\begin{equation}\label{equal2}
\alpha(t)\sim \|v(t)\|_{\dot{H}^1} \sim\|\tilde{u}(t)\|_{\dot{H}^1}.
\end{equation}
Finally, combining \eqref{equal1}) and \eqref{equal2} we obtain
\begin{align*}
\delta(u(t))&=\bigl|\|W\|^{2}_{\dot{H}^{1}}-\|\bigl[1+\alpha(t)\bigr]W+\tilde{u}(t)\|^{2}_{\dot{H}^{1}}\bigr|\\
&=2|\alpha| \|W\|^{2}_{\dot{H}^{1}}+O(\alpha^{2}),
\end{align*}
which shows $\delta(u(t))\sim |\alpha(t)|$. Putting it all together, we have
\[
\alpha(t)\sim \|v(t)\|_{\dot{H}^{1}}\sim \|\tilde{u}(t)\|_{\dot{H}^{1}}\sim \delta(u(t)).
\]

We turn to the proof of \eqref{Mod2}. We use the notation 
\[
U(t,x):=u_{[\theta(t), \lambda(t), x(t)]}(t,x)=e^{i\theta(t)}\tfrac{1}{\lambda(t)}u\bigl(t, \tfrac{x-x(t)}{\lambda(t)}\bigr).
\]

We perform a change of variables $t=t(s)$ such that 
\[
\tfrac{dt}{ds}=\tfrac{1}{\lambda(t)^2}.
\]
Then to establish \eqref{Mod2}, it suffices to prove that
\begin{equation}\label{changevar}
\bigl|x_s(s)-\tfrac{\lambda_s}{\lambda}(s)x_s(s)\bigr|+|\alpha_s(s)|+|\theta_s(s)|+\bigl|\tfrac{\lambda_s}{\lambda}(s)\bigr|\lesssim\delta(u(s)).    
\end{equation}

By direct computation, we derive
\begin{align*}
\partial_s U(t,x)
%&=e^{i\theta} \bigl\{-\tfrac{1}{\lambda^{2}} u\bigl(t,\tfrac{x - x(t)}{\lambda} \bigr) 
%+ \tfrac{1}{\lambda} \tfrac{dt}{ds} \, u_t\bigl(t,\tfrac{x - x(t)}{\lambda} \bigr)+ \tfrac{1}{\lambda} \partial_s \bigl( \tfrac{x - x(t)}{\lambda(t)} \bigr) \cdot \nabla u\bigl(t, \tfrac{x - x(t)}{\lambda} \bigr) \bigr\} + i\theta_s U \\
%&=-\tfrac{\lambda_s}{\lambda} U(t,x) 
%+ \tfrac{e^{i\theta}}{\lambda^{3}} u_t\bigl(t, \tfrac{x - x(t)}{\lambda} \bigr) - e^{i\theta} \tfrac{\lambda_s}{\lambda^{3}} (x - x(t)) \cdot \nabla u\bigl(t, \tfrac{x - x(t)}{\lambda} \bigr) + i\theta_s U \\ 
%&\,- e^{i\theta} \tfrac{\partial_s x(t(s))}{\lambda^{2}} \cdot \nabla u\bigl(t, \tfrac{x - x(t)}{\lambda} \bigr) \\[1ex]
&=-\tfrac{\lambda_s}{\lambda} \bigl( U(t,x) + (x - x(t)) \cdot \nabla U(t,x) \bigr) + \tfrac{e^{i\theta}}{\lambda^{3}} u_t\bigl(t, \tfrac{x - x(t)}{\lambda} \bigr) \\
&\quad - x_s \cdot \nabla U(t,x) + i\theta_s U,
\end{align*}
so that
\begin{equation}\label{f2.21}
\begin{aligned}
i\partial_s U+\Delta U 
%&=-i\tfrac{\lambda_s}{\lambda}\bigl(U(t,x)+(x-x(t))\cdot \nabla U(t,x)\bigr)\nonumber \\
%&\quad +\frac{e^{i\theta}}{\lambda^{3}}(iu_t+\Delta u)\bigl(t,\tfrac{x-x(t)}{\lambda}\bigr)-ix_s\cdot \nabla U(t,x)-\theta_s U\nonumber \\
%&=-i\tfrac{\lambda_s}{\lambda}\bigl(U(t,x)+(x-x(t))\cdot \nabla U(t,x)\bigr)+\tfrac{e^{i\theta}}{\lambda^{3}}|u|^{2}u\bigl(t,\frac{x-x(t)}{\lambda}\bigr)-ix_s\cdot \nabla U(t,x)-\theta_s U \nonumber \\
&= -|U|^{2}U -i\tfrac{\lambda_s}{\lambda}\bigl(U(t,x)+(x-x(t))\cdot \nabla U(t,x)\bigr) \\
&\quad\quad -ix_s\cdot \nabla U(t,x)-\theta_s U.
\end{aligned}
\end{equation}
As above, we decompose $U$ as 
\begin{equation*}
U=[1+\alpha(t)]W+\tilde{u}:=W+v,\qtq{where} \tilde{u}:=g_1+ig_2\in \mathcal{A}^{\perp}.
\end{equation*} 

We can then reformulate \eqref{f2.21} in terms of $v$ as
\begin{equation}\label{f2.23}
\begin{aligned}
\partial_s v & + (\Delta + W^2) \Im v -i(\Delta+3W^2)\Re v\\
& \quad -i\theta_s W +\tfrac{\lambda_s}{\lambda}[W_1-x(s)\cdot \nabla W] +x_s\cdot\nabla W \\
& = - i \theta_s v - R(v) - \tfrac{\lambda_s}{\lambda}\bigl(v+(x-x(s))\cdot \nabla v\bigr) - x_s\cdot\nabla v,
\end{aligned}
\end{equation}
where
\begin{equation*}
R(v):=-i \left|W+v\right|^2(W+v)+iW^2+3iW^2\Re v-W^2\Im v.
 \end{equation*}
In terms of $g_1$ and $g_2$, this becomes 
\begin{equation}\label{f2.24}
\begin{aligned}
\partial_s g_1 & +i\partial_s g_2+\alpha_s W+(\Delta+W^{2})g_2-i(\Delta+3W^{2})g_1-2i\alpha W^{3} \\
&\quad-\theta_s iW+\tfrac{\lambda_s}{\lambda}W_1-\bigl(\tfrac{\lambda_s}{\lambda}x(t)-x_s\bigr)\cdot\nabla W\\
& =-R(v)+i\theta_s v-\tfrac{\lambda_s}{\lambda}\bigl(v+x\cdot \nabla v\bigr)+\bigl(\tfrac{\lambda_s}{\lambda}x(t)-x_s\bigr)\cdot \nabla v.%:=\mathcal{E}. %Why introduce \E? It's never used. 
\end{aligned}
\end{equation}

% \begin{equation*}
%  R(v):=-i \left|W+v\right|^2(W+v)+iW^2+3iW^2\Re v-W^2\Im v.
%  \end{equation*}

We now introduce
\begin{equation*}
\epsilon(s):=|\delta(u(s))|\,\bigl\{|\delta(u(s))|+|\theta_s(s)|+\bigl|\tfrac{\lambda_s}{\lambda}(s)\bigr|+\left|\tfrac{\lambda_s}{\lambda}x(t)-x_s\right|\bigr\}
\end{equation*}
and define the constants
\[
c_1:=\|W\|_{\dot{H}^1}^2,\quad \quad c_2:=\|W_1\|_{\dot{H}^1}^2.
\]

We now multiply both sides of \eqref{f2.24} by $\Delta W$, integrate, take the real part, and utilize the orthogonality conditions
\begin{equation}\label{f2.25}
(g_1, W)_{\dot{H}^1}=(W, W_1)_{\dot{H}^1}=(W, \partial_j W)_{\dot{H}^1}=0.    
\end{equation}
This yields
\begin{equation}\label{f2.26}
c_1\alpha_s=-(\Delta g_2, W)_{\dot{H}^1}-(W^{2}g_2, W)_{\dot{H}^1}+O(\epsilon(s)).
\end{equation}
Similarly, we multiply by $\Delta i W$, integrate, and take the imaginary part to obtain
\begin{equation}\label{f2.27}
c_1\theta_s=-(\Delta g_1, W)_{\dot{H}^1}-3(W^{2}g_1, W)_{\dot{H}^1}-2\alpha (W^{3}, W)_{\dot{H}^1}+O(\epsilon(s)).
\end{equation}
Next, multiplying by $\Delta W_1$ taking the imaginary part leads to
\begin{equation}\label{f2.28}
c_2\tfrac{\lambda_s}{\lambda}=-(\Delta g_2, W_1)_{\dot{H}^1}-(W^{2}g_2, W_1)_{\dot{H}^1}+O(\epsilon(s)).
\end{equation}
Finally, we multiply by $\Delta\partial_j W$ and take the real part.  Introducing the notation
\[
\lambda_j:=\|\partial_j W\|_{\dot{H}^1}^2, \quad \quad \beta_j:=\tfrac{\lambda_s}{\lambda}x_j(t)-\tfrac{\partial_s x_j(t)}{\lambda} \qtq{for} j\in\{1,2,3,4\}
\]
and recalling \eqref{f2.25}, we derive
\begin{equation}\label{f2.29}
\lambda_j \beta_j(s)=-((\Delta+W^{2})g_2, \partial_j W)_{\dot{H}^1}+O(\epsilon(s)).
\end{equation}

Combining the estimates \eqref{f2.26}--\eqref{f2.29}, we conclude
\begin{equation}\label{f2.30}
|\alpha_s|+|\theta_s|+\left|\tfrac{\lambda_s}{\lambda}\right|+\sum_j|\beta_j(s)|\lesssim \|g\|_{\dot{H}^1}+O(\epsilon(s))\lesssim\delta(u)+O(\epsilon(s)).
\end{equation}
By choosing $\delta_0$ sufficiently small, we arrive at
\begin{equation}\label{f2.31}
|\alpha_s|+|\theta_s|+\bigl|\tfrac{\lambda_s}{\lambda}\bigr|+\sum_j|\beta_j(s)|\lesssim \delta(u),
\end{equation}
which implies the desired estimate \eqref{changevar}.
\end{proof}

%%%%%%%%%%%%%%%%%%%%%%%%%%%%%%%%%%%
%%%%%%%%%%%%%%%%%%%%%%%%%%%%%%%%%%%
%%%%%%%%%%%%%%%%%%%%%%%%%%%%%%%%%%%
%%%%%%%%%%%%%%%%%%%%%%%%%%%%%%%%%%%
%%%%%%%%%%%%%%%%%%%%%%%%%%%%%%%%%%%
%%%%%%%%%%%%%%%%%%%%%%%%%%%%%%%%%%%
\section{Proof of a rigidity result}

We next prove a rigidity result, which asserts that if a subcritical threshold solution blows up and stays very close to the orbit of the ground state in one time direction, it must coincide exactly with the heteroclinic orbit $W^-$ (modulo symmetries). 

\begin{theorem}\label{t2.3} Suppose $u$ is a solution to \eqref{1.1} satisfying $E(u)=E(W)$ and $\|\nabla u(0)\|_{L^2}<\|\nabla W||_{L^2}$.  In particular, by Proposition~\ref{T:no ftb}, $u$ is global.

Suppose that $u$ blows up forward in time.  There exists $\eta_*>0$ sufficiently small such that the following holds:  If
\[
\sup_{t\in[0,\infty)} \delta(u(t)) \leq \eta_*,
\]
then there exist $\lambda_0>0$, $\theta_0\in\R$, $x_0\in\R^4$, and $T\in\R$ such that 
\begin{equation}\nonumber
u(t,x)=\lambda_0 e^{i\theta_0}W^{-}(\lambda_0^2(t-T), \lambda_0 (x+x_0)).
\end{equation}
\end{theorem}

Throughout the rest of this section, we suppose that $u$ is a solution satisfying the assumptions of Theorem~\ref{t2.3} with some $\eta_*>0$ (which will be taken sufficiently small below). In particular, by Lemma~\ref{modula} there exist continuous functions
\begin{equation*}
\theta : [0,\infty)\rightarrow \R,\quad x:  [0,\infty)\rightarrow \mathbb{R}^4 ,\quad \lambda: [0,\infty)\rightarrow \mathbb{R}_+, \quad \alpha: [0,\infty)\rightarrow \mathbb{R}_+
\end{equation*}
such that if 
\begin{equation}\label{decompz}
v(t):=\alpha(t)W+\tilde{u}(t)=u_{[\theta(t),\lambda(t),x(t)]}(t)-W,
\end{equation}
then
\begin{align}
&|\alpha(t)|\sim \|v(t)\|_{\dot{H}^1}\sim \|\tilde{u}(t)\|_{\dot{H}^1}\sim \delta(u(t)),\label{mo1}\\
&\bigl|{x^{\prime}(t)}-\tfrac{\lambda^{\prime}(t)}{\lambda(t)}x(t)\bigr|+|\alpha^{\prime}(t)|+|\theta^{\prime}(t)|+\bigl|\tfrac{\lambda^{\prime}(t)}{\lambda(t)}\bigr|\leq C\lambda(t)^2 \delta(u(t)). \label{mo2}
\end{align}

By applying a fixed translation, we may assume
\[
x(0)=0.
\]

We now run a bootstrap argument based on a localized virial estimate (incorporating the modulation analysis) to prove that the scale function and translation parameters are nearly constant.

\begin{proposition}\label{bootstrap} The following holds provided $\eta_*$ is sufficiently small.  On any interval $[0,T]$ such that
\[
\max_{t\in[0, T]}\lambda(t)\leq2\min_{t\in[0, T]}\lambda(t)\qtq{and} \max_{t\in[0, T]}|{x(t)}|\leq 1,
\]
we have
\begin{equation}
\max_{t\in[0, T]}\lambda(t)\leq\tfrac{3}{2}\min_{t\in[0, T]}\lambda(t),\quad  \max\limits_{t\in[0, T]}|{x(t)}|\leq \tfrac{1}{2}.
\end{equation}

Consequently,
\begin{equation}\label{bootse1}
\max_{t\in[0, \infty)}\lambda(t)\leq\tfrac{3}{2}\min_{t\in[0, \infty)}\lambda(t),\quad  \max_{t\in[0, \infty)}\left|{x(t)}\right|\leq \tfrac{1}{2}.
\end{equation}
\end{proposition}

\begin{proof} Let us introduce the notation
\[
\lambda_{\min}=\min_{t\in[0,T]}\lambda(t) \qtq{and} \lambda_{\max} = \max_{t\in[0,T]}\lambda(t). 
\]

We use a virial estimate that takes the modulation analysis into account. Let $\varphi$ be a radial function satisfying
 \begin{equation}\label{start1}
\varphi(x) = \left\{ 
\begin{aligned}
& |x|^2, \qquad & | x | \le 1, \\
& C_0, \qquad & | x | \ge 2.  
\end{aligned}
\quad\mbox{ with } \quad \left|\partial^{\alpha}\varphi(x)\right|\lesssim_{\alpha} |x|^{2-|\alpha|}\right. 
\end{equation}
for all multiindices $\alpha$. We also set
\[
\varphi_{R}(x)=R^2\varphi(\tfrac{x}{R}).
\]
We then define the Morawetz potential
\[
M_{R}(u(t))=2\Im \int_{\R^4} \nabla \varphi_R(x) \cdot\nabla u(t,x)\overline{u(t,x)}\, dx.
\]

We will first show that for any $R\geq 1$, we have
\begin{equation}\label{step1}
|M_R(u(t))| \lesssim R^2\delta(u(t)).
\end{equation}
We will then show that if we take $R$ to be
\[
R=\max_{t\in[0,T]} \tfrac{C}{\lambda(t)} = C\lambda_{\min}^{-1}% + \max_{t\in[0,T]}\tfrac{|x(t)|}{\lambda(t)}
\]
for some large (universal) $C\gg1$, then we can obtain the lower bound
\begin{equation}\label{step2}
\tfrac{d}{dt}M_R(u(t)) \geq \delta(u(t)).
\end{equation}

Combining \eqref{step1} and \eqref{step2} and applying the Fundamental Theorem of Calculus then yields
\[
\int \delta(u(t)) \,dt \lesssim \lambda_{\min}^{-2}[\delta(u(0))+\delta(u(T))]\lesssim \eta_*\lambda_{\min}^{-2} .
\]
Applying the bound on $|\tfrac{\lambda'(t)}{\lambda(t)}|$ in \eqref{mo2} and the bootstrap assumption, this yields
\[
\int_0^T \bigl|\tfrac{d}{dt}\log \lambda(t)\bigr|\,dt \lesssim \eta_* \lambda_{\min}^{-2}\lambda_{\max}^{2} \lesssim \eta_*. 
\]
Choosing $\eta_*$ sufficiently small (depending only on universal constants), this implies $\lambda_{\max}\leq\tfrac32\lambda_{\min}$. 

Similarly, applying the upper bound on $|x'-\tfrac{\lambda'}{\lambda} x|$ and the bootstrap assumptions, we obtain
\begin{align*}
\int_0^T |x'(t)|\,dt & \leq \int_0^T |x'(t)-\tfrac{\lambda'(t)}{\lambda(t)}x(t)|\,dt + \int_0^T |\tfrac{\lambda'(t)}{\lambda(t)}|\, |x(t)|\,dt  \\
& \lesssim \lambda_{\max}^2 \lambda_{\min}^{-2} \eta_* \lesssim \eta_*,
\end{align*}
which (choosing $\eta_*$ sufficiently small) implies $\max_{t\in[0,T]}|x(t)| \leq \tfrac12$. 

Thus it remains to prove \eqref{step1} and \eqref{step2}.  In what follows we use the notation
\[
\tilde W(t):=\lambda(t) e^{-i\theta(t)}W(\lambda(t)x+x(t)).
\]

We first prove \eqref{step1}.  First, recalling that $W$ is real-valued, we note that
\[
M_R(\tilde W(t))\equiv 0.
\]
We will also use the fact that
\[
\| u\|_{L_t^\infty \dot H_x^1} \leq \|W\|_{\dot H_x^1}.
\]
Using H\"older's inequality and Sobolev embedding, we estimate
\begin{align*}
|M_R(u(t))| & = |M_R(u(t)) - M_R(\tilde W(t))| \\
& \lesssim R\int_{|x|\lesssim R} [|\nabla u| |u| - |\nabla \tilde W| |\tilde W|]\,dx \\
& \lesssim R^2 \{\|W\|_{\dot H^1}+\| u\|_{L_t^\infty \dot H_x^1}\}\{\|u-\tilde W\|_{L_t^\infty \dot H_x^1} + \|u-\tilde W\|_{L_t^\infty L_x^4}\} \\
& \lesssim R^2 \|u-\tilde W\|_{L_t^\infty \dot H_x^1}\lesssim R^2 \delta(u(t)). 
\end{align*}

We turn to the proof of \eqref{step2}. A tedious but straightforward calculation using \eqref{1.1} and integration by parts yields
\begin{align*}
\tfrac{d}{dt} M_R(u(t)) & = \int [-\Delta\Delta \varphi_R]|u|^2 - \Delta[\varphi_R] |u|^4 + 4\Re \bar u_j u_k \partial_{jk}[\varphi_R]\,dx \\
& =: 8\int |\nabla u|^2\,dx - 8\int |u|^4\,dx + F_R(u) \\
& = 8\delta(u(t)) + F_R(u),
\end{align*}
where 
\begin{align*}
F_R(u) & = -8\int_{|x|\geq R}|\nabla u|^2+8\int_{|x|\geq R}|u|^4\,dx \\
& \quad +\int_{R\leq|x|\leq 2R} [-\Delta\Delta\varphi_R]|u|^2 \,dx - \int_{|x|\geq R} \Delta[\varphi_R]|u|^4 \\
& \quad + 4\Re\int_{|x|\geq R} \bar u_j u_k \partial_{jk}[\varphi_R]\,dx.
\end{align*}
To exhibit $\delta(u(t))$, we have used the fact that $E(u)=E(W)=\tfrac14\|W\|_{\dot H^1}^2$. 

We next observe that 
\[
F_R(\tilde W(t))\equiv 0.
\]
for any $R\geq 1$.  Indeed, for fixed $t_0\in\R$ we may view $\tilde W(t_0)$ as a static solution to \eqref{1.1}.  Thus (by the same computation above and Pohozaev identities) we have trivially that $\tfrac{d}{dt} M_R(\tilde W(t_0))=F_R(\tilde W(t_0))$ for all $t,t_0\in\R$.  However, since $W$ is real-valued we have $M_R(\tilde W(t))\equiv 0$, which yields the claim. 

Thus we may write
\[
F_R(u(t)) = F_R(u(t)) - F_R(\tilde W(t)).
\]
We now wish to show that each term appearing in this difference can be controlled by a small multiple of $\delta(t)$, provided $R$ is chosen appropriately.  Let $\eta>0$.  We begin by changing variables to obtain
\begin{align*}
\biggl|\int_{|x|\geq R} & |\nabla u(t)|^2 - |\nabla \tilde W(t)|^2 \,dx\biggr| \\
& \lesssim\bigl\{ \|\nabla u(t)\|_{L^2(|x|\geq R)} + \|\nabla W\|_{L^2(|x-x(t)|\geq R\lambda(t)} \bigr\} \|\nabla v(t)\|_{L^2} 
\end{align*}
We now note that 
\[
\{|x-x(t)|>R\lambda(t)\} \subset \{|x|\geq \tfrac12 C\}.  
\]
Indeed, this follows from the fact that $|x(t)|\leq 1$ and the definition of $R$. Thus we can choose $C$ sufficiently large to obtain
\[
\|\nabla W\|_{L^2(|x-x(t)|\geq R\lambda(t))}\leq \eta. 
\]
On the other hand, recalling \eqref{decompz} and changing variables,
\[
\|\nabla u\|_{L^2(|x|\geq R)} \leq \|\nabla W\|_{L^2(|x-x(t)|\geq R\lambda(t))}+\|\nabla v(t)\|_{L_x^2} \lesssim \eta+\delta(u(t)).
\]
Thus 
\[
\biggl|\int_{|x|\geq R} |\nabla u(t)|^2 - |\nabla \tilde W(t)|^2 \,dx\biggr| \lesssim [\eta+\eta_*] \delta(u(t)). 
\]

The remaining terms in $F_R$ are treated in a similar manner. Let us demonstrate how to estimate one additional term.  Using H\"older's inequality and Sobolev embedding, 
\begin{align*}
\biggl| \int_{|x|\geq R} & |u|^4 \,dx - \int_{|x|\geq R} |\tilde W|^4\,dx \biggr| \\
& \lesssim \bigl\{ \|u\|_{L_x^4(|x|\geq R)}^3 + \|W\|_{L_x^4(|x-x(t)|\geq R\lambda(t))}^3\} \|v(t)\|_{L_x^4} \\
& \lesssim [\eta^3+\eta_*^3]\|v(t)\|_{\dot H^1} \lesssim [\eta^3+\eta_*^3]\delta(u(t)). 
\end{align*}
Estimating similarly for the remaining terms and using properties of $\varphi$, we obtain
\[
|F_R(u(t))| \lesssim [\eta+\eta_*]\delta(u(t)).
\]
Taking $\eta>0$ sufficiently small, this completes the proof of \eqref{step2} and hence of Proposition~\ref{bootstrap}.\end{proof}

We now prove that the modulation parameters converge exponentially. 

\begin{proposition}[Exponential convergence]\label{exp} There exist $\lambda_0>0$, $x_0\in\R^4$, $\gamma_0\in\R^4$, and ${c}_0>0$ such that
\begin{equation}\label{exp-convergence} 
\begin{aligned}
&|x(t)-x_0|+|\tfrac{\lambda(t)}{\lambda_0}-1|+|\theta(t)-\theta_0|\lesssim e^{-{c}_0 t},\\
& \alpha(t)\sim\delta(u(t))\lesssim e^{-{c}_0 t}.
\end{aligned}
\end{equation}

\end{proposition}

\begin{proof}

Applying Proposition~\ref{bootstrap} and recalling that we have normalized to $x(0)=0$, we have
\[
\lambda(t)\sim \lambda(0) \qtq{and} |x(t)| \leq 1\qtq{for all}t\geq 0.
\]
We now repeat the modulated virial estimate used to prove Proposition~\ref{bootstrap} to obtain the estimate
\[
\int_{t_1}^{t_2} \delta(u(t))\,dt \lesssim \delta(u(t_1))+\delta(u(t_2))
\]
for any $t_2>t_1\geq 0$. In particular, $t\mapsto \delta(u(t))\in L_t^1((0,\infty))$, so that there exists $T_n\to\infty$ such that $\delta(u(T_n))\to 0$.  Applying the estimate above on intervals $[t,T_n]$ thus yields
\[
\int_t^\infty \delta(u(s))\,ds \lesssim \delta(t) \qtq{for any}t\geq 0.
\]
By Gronwall's inequality, it follows that
\[
\int_t^\infty \delta(u(s))\,ds \lesssim e^{-c_0t}.
\]
The desired exponential convergence of modulation parameters now follows from \eqref{mo1}--\eqref{mo2}.\end{proof}

To complete the proof of Theorem~\ref{t2.3} requires one final ingredient:
\begin{proposition}\label{2WW}
Let $u$ be a (global) solution to \eqref{1.1} satisfying 
\[
E(u)=E(W)\qtq{and}\|u_0\|_{\dot{H}^1(\R^4)}<\|W\|_{\dot{H}^1(\R^4)}.
\]
Suppose further that  
\begin{equation}\label{decayf}
\left\|u(t)-W\right\|_{\dot{H}_x^1} \leq Ce^{-ct},\quad \forall \ t\geq 0
\end{equation}
for some $c,C>0$. Then there exists $T\in \mathbb{R}$ such that $u(t)=W^{-}(t-T).$
\end{proposition}
This result was established in \cite{duyckaerts1} in the radial setting. We will describe below how to obtain the non-radial analogue.  Taking Proposition~\ref{2WW} for granted for the moment, let us complete the proof of Theorem~\ref{t2.3}.

\begin{proof}[Proof of Theorem \ref{t2.3}]
By using the exponential convergence of the modulation parameters (Proposition~\ref{exp}) and the triangle inequality, we can readily obtain
\begin{align*}
\|u(t)-\lambda_0 e^{i\theta_0} W(\lambda_0 x + x_0)\|_{\dot H^1} & \lesssim \|u(t)-\lambda(t) e^{i\theta(t)}W(\lambda(t)x+x(t))\|_{\dot H^1} + e^{-c_0t}  \\
& \lesssim \delta(u(t)) + e^{-c_0 t} \lesssim e^{-c_0t}. 
\end{align*}
By a change of variables, this yields
\[
\big\|\tfrac{1}{\lambda_0}e^{-i\theta_0} u(\tfrac{t}{\lambda_0^2},\tfrac{x-x_0}{\lambda_0}) - W\big\|_{\dot H^1} \lesssim e^{-c_0\lambda_0^{-2} t}.
\]
We can therefore apply Proposition~\ref{2WW} to conclude that 
\[
\tfrac{1}{\lambda_0} e^{-i\theta_0}u(\tfrac{t}{\lambda_0^2},\tfrac{x-x_0}{\lambda_0}) = W^{-}(t-T)
\]
for some $T\in\R$, which yields Theorem~\ref{t2.3}.
\end{proof}

It therefore remains to prove Proposition~\ref{2WW}.

\subsection{Proof of Proposition~\ref{2WW}}\label{S:2WW}The proof of Proposition~\ref{2WW} largely follows the arguments of \cite[Sections 5--6]{duyckaerts1}.  Thus in this subsection, we will primarily sketch the strategy outlined in \cite{duyckaerts1} and describe the changes needed to address the non-radial setting.

We will write either $f=f_1+if_2$ or $f={f_1\choose f_2}$ for a complex-valued function $f$ with
real part $f_1$ and imaginary part $f_2$.  We suppose $u$ is a solution as in the statement of Proposition~\ref{2WW} and define $v=u-W$. Equation \eqref{1.1} yields
\[
\partial_t v + \mathcal{L}(v) + R(v) = 0,
\]
where $\mathcal{L}$ is the linearized operator
\begin{equation}\label{defL}
\mathcal{L}:=\left[\begin{array}{cc} 0 & \Delta+W^{2} \\ -\Delta-3W^2 & 0 \end{array}\right]
\end{equation}
and $R(v)$ is a the nonlinear term
\[
R(v):=-i|W+v|^2(W+v)+iW^3 + 3iW^2 v_1 - W^2 v_2. 
\]
We note that $\mathcal{L}$ has a pair of radial eigenvectors $e_+=\bar e_-$ with $\mathcal{L}e_\pm = \pm\lambda_1 e_\pm$ and $\lambda_1>0$.  

One key ingredient in the proof is the construction in \cite{duyckaerts1} of the special solution $W^-$ itself. This is accomplished by means of an iterative argument.  We summarize the construction as follows:  

\begin{remark}[The construction of $W^\pm$] Fix a parameter $a\in\R$.  Then there exist (radial) Schwartz functions $\{\Phi_j^a\}_{j\geq 1}$ with $\Phi_1^a = a e_+$ such that the functions
\[
W_k^a(t,x) = W(x) + \sum_{j=1}^k e^{-j\lambda_1 t}\Phi_j^a(x)
\]
are approximate solutions to \eqref{1.1} as $t\to\infty$, in the sense that
\[
(i\partial_t+\Delta)W_k^a = -|W_k^a|^2 W_k^a + \mathcal{O}(e^{-(k+1)\lambda_1 t}). 
\]
Applying a fixed point argument around these approximate solutions, one finds that for $k$ sufficiently large there is a unique radial solution $W^a$ to \eqref{1.1} on some interval $[t_k,\infty)$ that satisfies
\[
\|\nabla[W^a - W_k^a]\|_{L_t^6 L_x^{12/5}((t,\infty)\times\R^4)} \leq e^{-(k+\frac12)\lambda_1 t}.
\]
The solution $W^a$ is independent of $k$, and in particular satisfies 
\[
\|W^a - W - ae^{-\lambda_1 t}e_+\|_{\dot H^1} \leq e^{-\frac32\lambda_1 t}.
\]
The solution $W^-$ is obtained by taking $a=-1$.  Similarly, $W^+$ corresponds to $a=1$ and $W$ corresponds to $a=0$.  For $a\in\R\backslash\{0\}$, the solution $W^a$ is simply a time translation of $W^+$ or $W^-$ (depending on the sign of $a$).  

The solutions $W^\pm$ are distinguished by the fact that 
\[
\|W^-(t)\|_{\dot H^1}<\|W\|_{\dot H^1}<\|W^+(t)\|_{\dot H^1}
\]
for large $t>0$.
\end{remark}

In light of this construction, what we really need to prove is that there exists $a\in\R$ such that $u=W^a$, so that the subcritical assumption then forces $u$ to be some time translate of $W^-$, as desired.  

As we will describe, the key step will be to prove the following estimates. First, we have the decay estimate
\begin{equation}\label{key-dm0}
\|v(t)\|_{\dot H^1} + \|\nabla v\|_{L_t^6 L_x^{12/5}((t,\infty)\times\R^4)}\lesssim e^{-\lambda_1 t},
\end{equation}
for all $t$ large. Second, there exists $a\in\R$ such that 
\begin{equation}\label{key-dm}
\|v(t)-a e^{-\lambda_t} e_+ \|_{\dot H^1} + \|\nabla[v(s)- a e^{-\lambda_1 s}e_+]\|_{L_t^6 L_x^{12/5}((t,\infty)\times\R^4)}\lesssim e^{-\frac32\lambda_1 t}
\end{equation}
for all $t$ large. 

We begin by collecting what we know about $v$ so far.  Firstly, the assumptions in Proposition~\ref{2WW} and some further arguments (in the spirit of local well-posedness) allow us to prove the following estimates on $v$ and the nonlinear term $R(v)$:
\begin{align*}
&\|v(t)\|_{\dot H^1} + \|\nabla v\|_{L_t^6 L_x^{\frac{12}{5}}((t,\infty)\times\R^4)} \lesssim e^{-ct}, \\
&\|R(v(t))\|_{L_x^{\frac43}} + \|\nabla R(v)\|_{L_t^2 L_x^{\frac43}((t,\infty)\times\R^4)}\lesssim e^{-2ct}.
\end{align*}
Without loss of generality, we may assume $c<\lambda_1$.  We begin by improving the decay estimate on $v$ until we reach the decay rate $e^{-\lambda_1 t}$.  This is done as follows: 

We would like to use the equation that $v$ obeys.   In particular, it is useful at this point to give a spectral decomposition of $v$ relative to the linearized operator $\mathcal{L}$.  This is essentially the only place where the non-radial assumption plays a role, as it leads to a larger kernel for $\mathcal{L}$.  In particular, although $\mathcal{L}$ does not have any additional nonzero eigenvalues in the non-radial setting (cf. Lemma~\ref{lem.positivity} below), we must include the non-radial functions $\partial_j W$ for $j\in\{1,2,3,4\}$ in the kernel.  Recalling from Section~\ref{S:modulation} that $W_1:=W+x\cdot\nabla W$, we then decompose
\begin{equation}\label{vdecompp}
v(t) = \alpha_+(t) e_+ + \alpha_-(t) e_- + i\beta(t) W + \gamma_0(t) W_1 +\sum_{j=1}^4 \gamma_j(t) \partial_j W + v^\perp(t),
\end{equation}
so that $v^\perp$ belongs to a subspace on which the quadratic form 
\[
\F(f,g):=\tfrac12\Im\int \mathcal{L}f\,\bar g\,dx
\]
is positive definite.  Note that in the previous section, we have written $\F(f)=\F(f,f)$; we will continue to abuse notation slightly in this way.  The positivity of $\F$ relies on the following lemma, the proof of which follows the proof of \cite[Lemma~5.2]{duyckaerts1}.

\begin{lemma} \label{lem.positivity} Let $\mathcal{B}^\perp$ denote the set of functions $v\in \dot H^1$ such that
\[
(iW,v)_{\dot H^1} = (W_1,v)_{\dot H^1} = (\partial_j W,v)_{\dot H^1} = \F(e_+,v)=\F(e_-,v) = 0
\]
for $j\in\{1,2,3,4\}$. Then there exists $c>0$ such that
\begin{equation}\label{positif}
\F(f,f)\geq c\|f\|_{\dot H^1}^2 \qtq{for all}f\in \mathcal{B}^\perp. 
\end{equation}
\end{lemma}

\begin{proof}[Sketch of the proof] Let $f\in\mathcal{B}^\perp$. Recalling the set $\mathcal{A}$ from Section~\ref{S:modulation}, the starting point is to decompose $f$ and $e_\pm$ relative to this space.  In particular, we write
\[
f=\alpha W + \tilde h,\quad e_\pm = \pm\eta i W + \xi W_1 + \zeta W + h_\pm,
\]
with $\tilde h,h_\pm\in \mathcal{A}^\perp$ and $h_-=\bar h_+$.  Using the fact that 
\[
(\Delta + W^2)W^3 \notin\text{span}\{W,W_1,\partial_1W,\dots,\partial_4 W\},
\]
we can follow the arguments of \cite[Lemma~5.2]{duyckaerts1} to deduce the following:  First, $h_+$ and $h_-$ are linearly independent.  Second, we can write
\[
\F(f) = -\frac{\F(h_+,\tilde h)\F(h_-,\tilde h)}{\sqrt{\F(h_+)}\sqrt{\F(h_-)}} + \F(\tilde h),
\]
and finally there exists $b\in(0,1)$ such that
\[
\biggl| \frac{\F(h_+,g)\F(h_-,g)}{\sqrt{\F(h_+)}\sqrt{\F(h_-)}} \biggr| \leq b \F(g) \qtq{for all}g\in \mathcal{A}^\perp. 
\]
Thus (recalling Lemma~\ref{coer})
\begin{equation}\label{fts01}
\F(f) \geq (1-b)\F(\tilde h) \gtrsim \|\tilde h\|_{\dot H^1}^2.
\end{equation}
Now note that 
\[
\F(W)=-\|W\|_{\dot H^1}^2,\qtq{so that}\F(f) = \F(\tilde h)-\alpha^2\|W\|_{\dot H^1}^2.
\]
In particular, we have from \eqref{fts01} that $\alpha^2 \|W\|_{\dot H^1}^2 \leq b\F(\tilde h)$, so that
\begin{align*}
\|\tilde h\|_{\dot H^1}^2 & = \|f\|_{\dot H^1}^2 - \alpha^2\|W\|_{\dot H^1}^2 \\
& \geq \|f\|_{\dot H^1}^2 - b\F(\tilde h) \\
& \geq \|f\|_{\dot H^1}^2 - b[\F(f)+\alpha^2\|W\|_{\dot H^1}^2] \\
& \geq (1-b)\|f\|_{\dot H^1}^2 - b\F(f). 
\end{align*}
Inserting this back into \eqref{fts01}, we can obtain $\F(f)\gtrsim \|f\|_{\dot H^1}^2$, as desired.  \end{proof}

The equation for $v$ now leads to differential identities for the coefficient functions in \eqref{vdecompp} that may be used to estimate their size.  Using the assumptions on $v$ (and estimates on $R(v)$), one firstly obtains the following estimates on the coefficients $\alpha_\pm$: 
\[
|\alpha_-(t)| \lesssim e^{-ct},\quad |\alpha_+(t)|\lesssim \begin{cases} e^{-2ct} & c>\lambda_1, \\ \max\{e^{-2ct},e^{-\lambda_1t}\} & c\leq\lambda_1.\end{cases}
\]
For the estimates on the remaining coefficients, one obtains
\[
|\beta(t)|+\sum_{j=0}^4 |\gamma_j(t)| + \|v^\perp(t)\|_{\dot H^1} \lesssim e^{-\frac{3c}{2}t}.
\]
Thus we can obtain the estimate
\[
\|v(t)\|_{\dot H^1} \lesssim \max\{e^{-\frac{3c}{2}t},e^{-\lambda_1 t}\}
\]
(and we can obtain a similar estimate for the Strichartz norm). 

If $\tfrac{3c}{2}>\lambda_1$, we already obtain the desired estimate in \eqref{key-dm0}.  Otherwise, we can also obtain a decay rate of $e^{-3ct}$ for $R(v(t))$ and then iterate the argument just given, leading to
\[
\|v(t)\|_{\dot H^1} \lesssim \max\{e^{-\frac{9c}{4}t},e^{-\lambda_1 t}\}. 
\]
Iterating this finitely many times, we can finally obtain the desired decay rate of $e^{-\lambda_1 t}$ in \eqref{key-dm0}.

At this point, we modify the argument above slightly to obtain \eqref{key-dm}.  In particular, one can show that $v(t)-\alpha_+(t) e_+$ admits a better estimate than $v(t)$ itself.  Moreover, $\alpha_+(t)$ is asymptotic to $a e^{-\lambda_1 t}$ for some $a\in\R$, with a quantitative rate of convergence.  In particular, one can obtain
\begin{equation}\label{v-a-bd}
\|v(t)-a e^{-\lambda_1 t}e_+\|_{\dot H^1} \lesssim e^{-\frac32\lambda_1 t}. 
\end{equation}
Combining this with some local well-posedness arguments, we obtain \eqref{key-dm}.

Using \eqref{key-dm} as a base case, we would now like to show that for any $m\geq 0$, there exists $t_0>0$ such that
\[
\|u(t)-W^a(t)\|_{\dot H^1} \lesssim e^{-mt}
\]
for any $t\geq t_0$. This in turn yields $u=W^a$ by the uniqueness properties of $W^a$.

The estimate \eqref{key-dm} (and the properties of the construction of $W^\pm$) gives the base case $m=\tfrac32\lambda_1.$  To complete the inductive step, one notes that writing $v^a:=W^a(t)-W$, we have
\[
\partial_t(v-v^a) + \mathcal{L}(v-v^a) = -[R(v)-R(v^a)]. 
\]
This puts us in a position to apply the inductive hypothesis and an iterative argument much like the one we just described.  This allows us to improve the decay rate, which completes the induction and therefore completes the proof of Proposition~\ref{2WW}.  For more details, we refer the reader to \cite[Sections~5--6]{duyckaerts1}.

%%%%%%%%%%%%%%%%%%%%%%%%%%%%%%%%%%%%%%%%%%%%%%%%%%%%%%%%%%%%%%
%%%%%%%%%%%%%%%%%%%%%%%%%%%%%%%%%%%%%%%%%%%%%%%%%%%%%%%%%%%%%%
%%%%%%%%%%%%%%%%%%%%%%%%%%%%%%%%%%%%%%%%%%%%%%%%%%%%%%%%%%%%%%
%%%%%%%%%%%%%%%%%%%%%%%%%%%%%%%%%%%%%%%%%%%%%%%%%%%%%%%%%%%%%%
\section{Proof of the main result}\label{reduction}

In this section we put together the pieces assembled above to complete the proof of our main result, Theorem~\ref{t1.1}. 

We begin with a lemma that controls the variation of the functional $\delta$ on an interval by the scattering size of that interval.

\begin{lemma}\label{scattersmall} Let $u:\R\times\R^4\to\C$ be a subcritical threshold solution and $J\subset\R$ an interval such that
\[
\|u\|_{L_{t,x}^6(J\times\R^4)}\leq M.
\]
Then there exist $0<c(M)\ll1$ and $C(M)\geq 1$ such that
\begin{equation}\label{f3.2}
\inf_{t\in J}\delta(u(t)) \leq c(M) \implies \sup_{t\in J}\delta(u(t))\leq C(M)\inf_{t\in J}\delta(u(t)).
\end{equation}
\end{lemma}

\begin{proof} We first let $c(M)<\tfrac12\delta_0$, with $\delta_0$ as in Lemma~\ref{modula}. We choose $t_0\in J$ such that 
\begin{equation}\label{equ:delt0}
\delta(u(t_0))<2\inf\limits_{t\in J} \delta(u(t))<\delta_0.
\end{equation}
By Lemma~\ref{Mod2}, we may write
\[
u(t_0,x) = \tilde v(x)+\tilde W(x),
\]
where
\[
\tilde W(x):=\lambda(t_0)e^{i\theta(t_0)}W(\lambda(t_0)x+x(t_0))\qtq{and}\|\tilde{v}\|_{\dot{H}^1}\lesssim \delta(u(t_0)).
\]

Now let $\tilde{v}(t)$ and $\tilde W$ be given by
\begin{equation*}
\tilde{v}(t,x)=u(t,x)-\tilde W(x)%=u(t,x)-\lambda(t_0)e^{i\theta(t_0)}W(\lambda(t_0)x+x(t_0)).
\end{equation*}
Since $u(t)$ and $\tilde W$ both solve \eqref{1.1}, we have
\begin{align*}
(i\partial_{t}+\Delta) \tilde{v} = |\tilde W|^2 \tilde W - |u|^2 u.%=\left|\lambda(t_0)W(\lambda(t_0)x+x(t_0))\right|^2\lambda(t_0)e^{i\theta_0}W(\lambda(t_0)x+x(t_0))-\left|u(t,x)\right|^2u(t,x).
\end{align*}
Observing that
\[
|\nabla [|\tilde W|^2 \tilde W - |u|^2u ]| \lesssim |\nabla \tilde v|\,|u|^2 + |\tilde v|\,|\nabla\tilde W|\,\{|u|+|\tilde W|\},
\]
we employ Strichartz estimates and Sobolev embedding to obtain
\begin{align*}
\|\nabla \tilde v\|_{S^0} & \lesssim \|\tilde v(t_0)\|_{\dot H^1} + \|\tilde v\|_{L_t^\infty \dot H_x^1}\|u\|_{L_t^3 L_x^{12}}^2 \\
&\quad + \|W\|_{\dot H^1} \bigl\{ \|u\|_{L_t^3 L_x^{12}}\|\tilde v\|_{L_t^3 L_x^{12}} + \|\tilde v\|_{L_t^3 L_x^{12}}^2\bigr\} \\
& \lesssim \|\tilde v(t_0)\|_{\dot H^1} + \|\nabla \tilde v\|_{S^0}\|u\|_{L_t^3 L_x^{12}}^2 + \|u\|_{L_t^3 L_x^{12}}\|\nabla \tilde v\|_{S^0}+ \|\nabla\tilde v\|_{S^0}^2,
\end{align*}
where all space-time norms are taken over $J\times\R^4$.  Using the fact that $\|u\|_{L_{t,x}^6}\leq M$ (which implies $\|\nabla u\|_{S^0}\lesssim_M 1$ via the local theory), we can split the interval $J$ into finitely many subintervals on which the $L_t^3 L_x^{12}$-norm of $u$ is small and run a standard bootstrap argument to obtain
\[
\|\nabla \tilde v\|_{S^0(J)} \lesssim_M \|\tilde v(t_0)\|_{\dot H^1}.
\]
Note that this requires choosing $\inf_{t\in J}\delta(u(t))$ sufficiently small depending on $M$.  Using this bound, we now readily obtain that
\begin{align*}
|\delta(u(t))|  & = \|\tilde W\|_{\dot H^1}^2 - \|u(t)\|_{\dot H^1}^2 \\
& \lesssim \|W\|_{\dot H^1} \|\tilde v(t)\|_{\dot H^1}  \lesssim_M \|\tilde v(t_0)\|_{\dot H^1}.
\end{align*}
Recalling \eqref{equ:delt0}, the result follows. \end{proof}

\begin{proof}[Proof of Theorem~\ref{t1.1}] Let $u_0\in \dot H^1$ satisfy 
\[
E(u_0)=E(W)\qtq{and}\|\nabla u_0\|_{L^2}<\|\nabla W\|_{L^2}.
\]
Let $u$ be the corresponding solution to \eqref{1.1} with $u|_{t=0}=u_0$.  By Proposition~\ref{T:no ftb}, $u$ is global. 

We assume that $u$ fails to scatter in both time directions, that is, $u$ blows up in at least one time direction.  To complete the proof, we must then show that $u=W^-$ modulo symmetries.

By Proposition~\ref{sequence}, there exists a sequence $\{t_n\}\subset\R$ such that $\delta(u(t_n))\to 0$. By Lemma~\ref{GlobalS} and continuity of the flow in $\dot H^1$ we must have $t_n\to\infty$ or $t_n\to-\infty$. 

We now observe that if $t_n\to\infty$, then $u$ blows up forward in time, and if $t_n\to-\infty$, then $u$ blows up backward in time.  Indeed, suppose $t_n\to\infty$ but $S_{[0,\infty)}(u)\leq M<\infty$.  Then, applying Lemma~\ref{scattersmall}, we obtain
\[
\delta(u_0)\leq \sup_{t\in[0,t_n]}\delta(u(t)) \lesssim_M \inf_{t\in[0,t_n]}\delta(u(t)) \to 0,
\]
which contradicts our assumptions on $u_0$.  The argument if $t_n\to-\infty$ is similar. 

Without loss of generality, we now assume $t_n\to\infty$ (and that $u$ blows up forward in time).  We split into two possible cases, depending on the small parameter $\eta_*$ introduced in Theorem~\ref{t2.3}.
\begin{itemize}
\item[(i)] There exists $t_0>0$ such that
\[
\sup_{t\in[t_0,\infty)}\delta(u(t)) \leq \eta_*.
\]
\item[(ii)] There exists an increasing sequence $t_n^-\to\infty$ such that
\[
\delta(u(t_n^-))>\eta_*\qtq{for all}n.
\]
\end{itemize}

In case (i), we may apply a time translation in order to assume $t_0=0$.  We can then directly apply Theorem~\ref{t2.3} to conclude that $u=W^-$ modulo symmetries and hence complete the proof.

It therefore remains to show that case (ii) cannot occur.  To this end, we assume case (ii) does occur and seek a contradiction.

First, passing to a subsequence, we may assume that $t_n$ is increasing,
\[
\delta(u(t_n)) < \eta_*, \qtq{and}t_n^- < t_n < t_{n+1}^- \qtq{for all} n.
\]
We now define the increasing sequence
\[
t_n^+ = \inf\bigl\{t\in\R: \sup_{\tau\in[t,t_n]} \delta(u(\tau))<\eta_*\bigr\} 
\]
and observe by continuity of $t\mapsto \delta(u(t))$ that 
\[
t_n^+ \to \infty\qtq{and} \delta(t_n^+)=\eta_*. 
\]

We now claim that
\begin{equation}\label{blowup+}
\lim_{n\to\infty} \|u\|_{L_{t,x}^6([t_n^+,t_n]\times\R^4)} = \infty. 
\end{equation}
Indeed, suppose instead that 
\[
\|u\|_{L_{t,x}^6([t_n^+,t_n]\times\R^4)} \leq M<\infty
\]
for all $n$. Then by Lemma~\ref{scattersmall} we have
\[
\eta_* = \delta(t_n^+)\lesssim_M \delta(t_n) \to 0,
\]
a contradiction. 

By Lemma~\ref{alcompact}, $u$ is almost periodic modulo symmetries on $[0,\infty)$.  Thus there exist $(\lambda_n,x_n)\in(0,\infty)\times\R^4$ and $w_0\in\dot H^1$ such that
\[
\lambda_n u(t_n^+,\lambda_n(\cdot+x_n)) \to w_0 \qtq{in}\dot H^1\qtq{as}n\to\infty. 
\]
Furthermore, we must have $\delta(w_0)=\eta_*$.  

We now let $w$ be the solution to \eqref{1.1} with $w|_{t=0}=w_0$. In particular, $w$ is a solution to \eqref{1.1} satisfying
\[
E(w)=E(W) \qtq{and}\|\nabla w_0\|_{L^2}<\|\nabla W\|_{L^2},
\]
so that (by Proposition~\ref{T:no ftb}) $w$ is global.  

We now claim that
\begin{equation}\label{final_claim}
S_{[0,\infty)}(w) = S_{(-\infty,0]}(w)=\infty,\qtq{while} \sup_{t\in[0,\infty)} \delta(w(t))\leq \eta_*.
\end{equation}
This then yields a contradiction to Theorem~\ref{t2.3}.  Indeed, using Theorem~\ref{t2.3}, we can conclude that $w=W^-$ modulo symmetries.  However, $w$ blows up in both time directions, while $W^-$ scatters in one time direction.

We turn to the proof of \eqref{final_claim}.  We define
\[
\tilde u_n(t,x) = \lambda_n u(t_n^+ + \lambda_n^2 t,\lambda_n(x+x_n))
\]

First, if $S_{[0,\infty)}(w_0)\leq M$, then by stability (Lemma~\ref{pertu}) we can obtain
\[
S_{[0,\infty)}(\tilde u_n) = S_{[t_n^+,\infty)}(u) \lesssim_M 1
\]
for all $n$ sufficiently large, contradicting \eqref{blowup+}.  Similarly, if $S_{(-\infty,0]}(w)\leq M$, then
\[
S_{(-\infty,0]}(\tilde u_n) = S_{(-\infty,t_n^+]}(u) \lesssim_M 1
\]
for all $n$ sufficiently large. As $t_n^+\to \infty$, this contradicts that $u$ blows up forward in time. 

Next, we fix $T>0$. By local well-posedness, we have that
\[
S_{[0,T]}(\tilde u_n) \lesssim_T 1,
\]
while by a change of variables we have
\[
S_{[0,\lambda_n^{-2}(t_n-t_n^+)]}(\tilde u_n)=S_{[t_n^+,t_n]}(u)\to\infty. 
\]
Thus we must have that $T<\lambda_n^{-2}(t_n-t_n^+)$ for all large $n$.  In particular, $t_n^++\lambda_n^2 T\in[t_n^+,t_n]$, so that by definition of $t_n^+$, the triangle inequality, a change of variables, and Lemma~\ref{pertu}, we have
\begin{align*}
\delta(w(T)) & \leq \delta(\tilde u_n(T)) + \|\tilde u_n(T)\|_{\dot H^1}^2 - \|w(T)\|_{\dot H^1}^2 \\
& \leq \delta(u(t_n^++\lambda_n^2 T)) + C\|u_n(T)-w(T)\|_{\dot H^1} \\
& \leq \eta_* + o(1)
\end{align*}
as $n\to\infty$. As $T>0$ was arbitrary, the second claim in \eqref{final_claim} follows. \end{proof}

%%%%%%%%%%%%%%%%%%%%%%%%%%%%%%%%%%%%%%%%%%%%%%%%%%%%%%%%%%%%%%
%%%%%%%%%%%%%%%%%%%%%%%%%%%%%%%%%%%%%%%%%%%%%%%%%%%%%%%%%%%%%%
%%%%%%%%%%%%%%%%%%%%%%%%%%%%%%%%%%%%%%%%%%%%%%%%%%%%%%%%%%%%%%
%%%%%%%%%%%%%%%%%%%%%%%%%%%%%%%%%%%%%%%%%%%%%%%%%%%%%%%%%%%%%%
%%%%%%%%%%%%%%%%%%%%%%%%%%%%%%%%%%%%%%%%%%%%%%%%%%%%%%%%%%%%%%

\appendix

\section{Higher dimensions}\label{S:highd}

In this section, we briefly describe the extension of Theorem~\ref{t1.1} to dimensions $d\geq 5$. Standard topics such as well-posedness and stability are well-established, although we note that in dimensions $d\geq 6$ there are technical difficulties due to the low power of the nonlinearity.  Fortunately, these technical issues have been previously addressed in works such as \cite{KV1, li1, campos, Visan2, Tao1}, in the context of well-posedness and stability as well as in the context of results such as Proposition~\ref{2WW}. 

The main difference between dimension $d=4$ and dimensions $d\geq 5$ is the fact that in dimensions $d\geq 5$, one can obtain additional decay for almost periodic solutions, which substantially simplifies the arguments needed to obtain Proposition~\ref{sequence} (the existence of a sequence $\{t_n\}$ such that $\delta(u(t_n))\to 0$).  As described in the proof of Proposition~\ref{sequence}, we can reduce matters to proving this result for a global almost periodic solution $u$ satisfying the hypotheses of Theorem~\ref{t1.1} and the condition $\inf_{t\in\R}N(t)\geq 1$.  With Proposition~\ref{sequence} result in hand, the rest of the argument follows along similar lines (taking into account the modifications needed to treat the lower power nonlinearity as in \cite{li1, campos}).  Thus in the rest of this section, we focus on proving the analogue of Proposition~\ref{sequence}.

Following the approach of \cite{KV1}, we can obtain the following properties for an almost periodic solution with $\inf_{t\in\R} N(t)\geq 1$ in the case $d\geq 5$:
\begin{itemize}
\item $u\in L_t^\infty H_x^{-\eps}(\R\times\R^d)$ for some $\eps>0$, 
\item By modifying the compactness modulus function appropriately, we may take $N(t) \equiv 1$,
\item $x(t)=o(t)$ as $t\to\infty$. 
\end{itemize}
We remark that these arguments are insensitive to the sign of the nonlinearity, and in particular can be applied to any almost periodic solution.

With these properties in hand, the standard localized virial argument suffices to prove the existence of a sequence $t_n\to\infty$ such that $\delta(u(t_n))\to 0$. % With this result in hand, the rest of the argument follows along similar lines (taking into account the modifications needed for lower powers, as in \cite{li1, campos}).

\begin{proposition} Let $u$ be a solution to \eqref{equ:encrinls} as above. There exists $t_n\to\infty$ such that $\delta(u(t_n))\to 0$.
\end{proposition}

\begin{proof}[Sketch of the proof] Let $\varphi(x)$ be a smooth radial funtion satisfying $\varphi(x)=|x|^2$ for $|x|\leq 1$ and $\varphi(x)=C_0$ for $|x|\geq 2$.  We set $\varphi_R(x) = R^2\varphi(\tfrac{x}{R})$ and 
\[
M_R(t) = 2\Im\int \nabla \varphi_R(x)\cdot \nabla u(t,x)\bar u(t,x)\,dx.
\]
Using $u\in L_t^\infty L_x^2$, we have
\[
\sup_{t\in\R}|M(t)|\lesssim R\|u\|_{L_t^\infty \dot H_x^1} \|u\|_{L_t^\infty L_x^2} \lesssim R. 
\]
Next, by direct computation using \eqref{equ:encrinls} and integration by parts, we find 
\begin{align*}
\tfrac{d}{dt} M_R(t) & = \int -\Delta\Delta\varphi_R |u|^2 - \tfrac{4}{d}\Delta[\varphi_R]|u|^4 + 4\Re \bar u_j u_k \partial_{jk}[\varphi_R]\,dx \\
& = \tfrac{16}{d-2}\delta(u(t)) + F_R(u(t))
\end{align*}
where
\[
F_R(u) = O\biggl( \int_{|x|\geq R} |\nabla w|^2 + |w|^{\frac{2d}{d-2}}\,dx + \biggl(\int_{R\leq |x|\leq 2R} |w|^{\frac{2d}{d-2}}\,dx\biggr)^{\frac{d-2}{2}}\biggr). 
\]
Now let $\eps,\eta>0$.  We let $T\gg1$ and let $R=\eps T$. By the fundamental theorem of calculus,
\begin{equation}\label{easy-vir}
\begin{aligned}
\int_0^T \delta(u(t))\,dt  & \lesssim \sup_{t\in[0,T]}|M_R(u(t))| + \int_0^T |F_R(u(t))|\,dt \\
& \lesssim \eps T + \int_0^T |F_R(u(t))|\,dt.
\end{aligned}
\end{equation}
As $x(t)=o(t)$, there exists $T=T_{\eps,\eta}$ large enough that 
\[
\eps t - |x(t)|\geq C(\eta) \qtq{for all}t\geq T_{\eps,\eta} 
\]
for all $t>T_{\eps,\eta}$, where $C(\eta)$ is the compactness modulus function.  In particular, since $|x|\geq R$ implies $|x-x(t)|\geq C(\eta)$ for $t\geq T_{\eps,\eta}$, we have
\[
|F_R(u(t))| \lesssim \begin{cases} 1 & t\in[0,T_{\eps,\eta}] \\ \eta^{\frac{d-2}{2}} & t\in[T_{\eps,\eta},T].\end{cases}
\]
Continuing from \eqref{easy-vir}, we obtain
\[
\int_0^T \delta(u(t))\,dt  \lesssim [\eps+\eta^{\frac{d-2}{2}}] T + T_{\eta,\eps}.
\]
From this inequality, we can now deduce the existence of a sequence $t_n\to\infty$ along which $\delta(u(t_n))\to 0$.\end{proof}


\begin{thebibliography}{9}

\bibitem{Au76} T. Aubin. \'{E}quations diff\'erentielles non lin\'eaires et probl\`eme de {Y}amabe concernant la courbure scalaire. J. Math. Pures Appl. (9), 55(3):269--296, 1976.

\bibitem{bourgain} J. Bourgain, \emph{Global wellposedness of defocusing critical NLS in the radial case}. J. Amer. Math. Soc., \textbf{12} (1999), 145--171. 

\bibitem{campos} L. Campos, L. Farah and S. Roudenko, \emph{Threshold solutions for the nonlinear Schr\"odinger equation}, Rev. Mat. Iberoam. {\bf 38} (2022),  1637--1708.

%\bibitem{cazenave} T. Cazenave and F. Weissler, \emph{The Cauchy problem for the critical nonlinear Schr\"odinger equation in $H^s$.} Nonlinear Anal. \textbf{14} (1990), 807--836. 

%\bibitem{colliander} J. Colliander, M. Keel, G. Staffilani, H. Takaoka, and T. Tao, \emph{Global existence and scattering for rough solutions of a nonlinear Schr\"odinger equation on $\R^3$}. Comm. Pure Appl. Math. \textbf{57} (2004), 987--1014. 

\bibitem{CKSTT} J. Colliander, M. Keel, G. Staffilani, H. Takaoka, and T. Tao, \emph{Global well-posedness and scattering for the energy-critical nonlinear Schr\"odinger equation in $\R^3$.} Ann. of Math. (2) \textbf{167}, no. 3, 767--865.

\bibitem{DodsonMC} B. Dodson, \emph{Global well-posedness and scattering for the mass critical nonlinear Schr\"odinger equation with mass below the mass of the ground state.} Adv. Math. \textbf{285} (2015), 1589--1618.



\bibitem{dodson1} B. Dodson, \emph{Global well-posedness and scattering for the focusing, cubic Schr\"odinger equation in dimension $ d = 4 $}, Ann. Sci. Ecole Norm. S., serie 4, \textbf{52} (2019), 139--180. 


%\bibitem{dodson3} B. Dodson, \emph{The $L^2$ sequential convergence of a solution to the one-dimensional, mass-critical NLS above the ground state}, SIAM J. Math. Anal. {\bf 53} (2021), no. 4, 4744--4764.


%\bibitem{dodson4} B. Dodson, \emph{The $L^2$ sequential convergence of a solution to the mass-critical NLS above the ground state}, Nonlinear Anal. {\bf 215} (2022), Paper No. 112612, 25 pp.


\bibitem{dodson5} B. Dodson, \emph{A determination of the blowup solutions to the focusing NLS with mass equal to the mass of the soliton}, Ann. PDE {\bf 9} (2023), no. 1, Paper No. 3, 86 pp.


\bibitem{dodson6} B. Dodson, \emph{A determination of the blowup solutions to the focusing, quintic NLS with mass equal to the mass of the soliton}, Anal. PDE {\bf 17} (2024), no. 5, 1693--1760.

\bibitem{dodson2} B. Dodson, C. Miao, J. Murphy and J. Zheng, \emph{The defocusing quintic NLS in four space dimensions, Ann. Inst. H. Poincar\'e{} C Anal. Non Lin\'eaire} {\bf 34} (2017),  759--787. 

%\bibitem{duyckaerts2} T. Duyckaerts, J. Holmer and S. Roudenko, \emph{Scattering for the non-radial 3D cubic nonlinear Schr\"odinger equation}, Math. Res. Lett. {\bf 15} (2008), no.~6, 1233--1250.

\bibitem{duyckaerts3} T. Duyckaerts and F. Merle, \emph{Dynamics of threshold solutions for energy-critical wave equation}, Int. Math. Res. Pap. IMRP {2008}, Art ID rpn002, 67 pp. 

\bibitem{duyckaerts1} T. Duyckaerts and F. Merle, \emph{Dynamic of threshold solutions for energy-critical NLS}, Geom. Funct. Anal. {\bf 18} (2009),  1787--1840.

\bibitem{duyckaerts} T. Duyckaerts and F. Merle, \emph{Scattering norm estimate near the threshold for energy-critical focusing semilinear wave equation}, Indiana Univ. Math. J. {\bf 58} (2009), no. 4, 1971--2001.

\bibitem{duyckaerts4} T. Duyckaerts and S. Roudenko, \emph{Threshold solutions for the focusing 3D cubic Schr\"odinger equation}, Rev. Mat. Iberoam. {\bf 26} (2010), no.~1, 1--56.

%\bibitem{duyckaerts5} T. Duyckaerts and S. Roudenko, \emph{Going beyond the threshold: scattering and blow-up in the focusing NLS equation}, Comm. Math. Phys. {\bf 334} (2015),  1573--1615.

%\bibitem{duyckaerts6} T. Duyckaerts, C. Kenig and F. Merle, \emph{Profiles for bounded solutions of dispersive equations, with applications to energy-critical wave and Schr\"odinger equations}, Commun. Pure Appl. Anal. {\bf 14} (2015), no. 4, 1275--1326.

%\bibitem{Fan} C. Fan, \emph{The $L^2$ weak sequential convergence of radial focusing mass critical NLS solutions with mass above the ground state}, Int. Math. Res. Not. IMRN (2021), no. 7, 4864--4906.

\bibitem{GinibreVelo} J. Ginibre and G. Velo, \emph{Smoothing properties and retarded estimates for some dispersive evolution equations.} Comm. Math. Phys. \textbf{144} (1992), 163--188.

%\bibitem{holmer} J. Holmer and S. Roudenko, \emph{A sharp condition for scattering of the radial 3D cubic nonlinear Schr\"odinger equation}, Comm. Math. Phys. {\bf 282} (2008), no.~2, 435--467.

\bibitem{KeelTao}  M. Keel and T. Tao, \emph{Endpoint Strichartz estimates.} Amer. J. Math. \textbf{120} (1998), no. 5, 955--980.

\bibitem{kenig-merle1} C. E. Kenig and F. Merle, \emph{Global well-posedness, scattering and blow up for the energy-critical, focusing, nonlinear Schr\"odinger equation in the radial case}, Invent. Math. \textbf{166} (2006), 645--675.

%\bibitem{kenig-merle2} C. E. Kenig and F. Merle, \emph{Global well-posedness, scattering and blow-up for the energy critical focusing non-linear wave equation}, Acta Math. {\bf 201} (2008),  147--212. 

\bibitem{keraani1} S. Keraani, \emph{On the defect of compactness for the Strichartz estimates for the Schr\"odinger equations}. J. Diff. Eq. \textbf{175} (2001), 353--392. 

%\bibitem{keraani2} S. Keraani, \emph{On the blow up phenomenon of the critical nonlinear Schr\"odinger equation.} J. Funct. Anal., \textbf{235} (2006), 171--192. 

\bibitem{KVnote} R. Killip and M. Vi\c{s}an, \emph{Nonlinear Schr\"odinger equations at critical regularity.} Lecture notes prepared for Clay Mathematics Institute Summer School, Z$\ddot{u}$rich, Switzerland, 2008.

\bibitem{KV1} R. Killip and M. Vi\c{s}an, \emph{The focusing energy-critical nonlinear Schr\"odinger equation in dimensions five and higher.} Amer. J. Math., \textbf{132} (2010), 361--424. 

\bibitem{KV2} R. Killip and M. Vi\c{s}an, \emph{Global well-posedness and scattering for the defocusing quintic NLS in three dimensions}, Anal. PDE {\bf 5} (2012),  855--885.

\bibitem{li1} D. Li and X. Zhang, \emph{Dynamics for the energy critical nonlinear Schr\"odinger equation in high dimensions}, J. Funct. Anal. {\bf 256} (2009),  1928--1961. 

%\bibitem{li2} D. Li and X. Zhang, \emph{Dynamics for the energy critical nonlinear wave equation in high dimensions}, Trans. Amer. Math. Soc. {\bf 363} (2011),  1137--1160. 

\bibitem{Olivier} O. Rey, \emph{The role of the Green's function in a nonlinear elliptic equation involving the critical Sobolev exponent}, J. Funct. Anal. {\bf 89} (1990), no. 1, 1--52.

\bibitem{Ryckman} E. Ryckman and M. Vi\c{s}an, \emph{Global well-posedness and scattering for the defocusing energy-critical nonlinear Schr\"odinger equation in $\R^{1+4}$}. Amer. J. Math. \textbf{129} (2007), 1--60. 

\bibitem{Strichartz} R. Strichartz, \emph{Restrictions of Fourier transforms to quadratic surfaces and decay of solutions of wave equations.} Duke Math. J. \textbf{44} (1977), no. 3, 705--714.

\bibitem{SuZhao} Q. Su and Z. Zhao, \emph{Dynamics of subcritical threshold solutions for energy-critical NLS.} Dyn. Partial Diff. Eq. \textbf{20} (2023), no. 1, 37--72. 

\bibitem{Ta76} G. Talenti, Best constant in {S}obolev inequality.  Ann. Mat. Pura Appl. (4), \textbf{110} (1976), 
353--372.

\bibitem{Tao1} T. Tao and M. Vi\c{s}an, \emph{Stability of energy-critical nonlinear Schr\"odinger equations in high dimensions}, Electron. J. Differential Equations {2005}, No. 118, 28 pp.

\bibitem{Visan2} M. Vi\c{s}an, \emph{The defocusing energy-critical nonlinear Schr\"odinger equation in higher dimensions.} Duke Math. J., \textbf{138} (2007) 281--374. 

\bibitem{Visan3} M. Vi\c{s}an, \emph{Global well-posedness and scattering for the defocusing cubic nonlinear Schr\"odinger equation in four dimensions.} Int. Math. Res. Not. (2012), 1037--1067. 

\bibitem{VisanOberwolfach}  M. Vi\c{s}an, \emph{Dispersive Equations.} In Dispersive equations and nonlinear waves. Oberwolfach Seminars, 45. Birkh\"auser/Springer, Basel, 2014. xii+312 pp.

\end{thebibliography}
\end{document}